\documentclass[12pt]{article}
\usepackage{amsmath,amsfonts,amssymb,amsthm}
\usepackage{bbm}
\usepackage{graphicx}
\usepackage{mathtools}
\usepackage{tabularx}
\usepackage{hyperref}
\usepackage{color}
\newcommand{\avint}{\int \hspace{-1.0em}-\,}
\newcommand{\omeganu}{\omega_{\nu}}
\newcommand{\xinu}{\xi_{\nu}}
\newcommand{\unu}{u_{\nu}}
        
\usepackage{xargs}                      
\usepackage[pdftex,dvipsnames]{xcolor}
\usepackage[colorinlistoftodos,prependcaption,textsize=footnotesize]{todonotes}
\newcommandx{\improvement}[2][1=]{\todo[linecolor=Plum,backgroundcolor=Plum!25,bordercolor=Plum,#1]{#2}}
\newcommandx{\priority}[2][1=]{\todo[linecolor=Plum,backgroundcolor=Plum!55,bordercolor=Plum,#1]{#2}}

\definecolor{darkgreen}{rgb}{0,0.55,0}
\newcommand{\xin}{\xi_{\nu}}
\newcommand{\vn}{v_{\nu}}
\newcommand{\un}{\unu}
\newcommand{\fn}{f_{\nu}}

\newcommand{\grad}{\nabla}

\newcommand{\laplace}{\Delta}

\renewcommand{\div}{\grad\cdot}

\newcommand{\N}{\mathbbm{N}}

\newcommand{\R}{\mathbbm{R}}

\renewcommand{\H}{\mathbbm{H}}

\newcommand{\Z}{\mathbbm{Z}}

\def\loc{{\mathrm{loc}}}
\def\weak{{\mathrm{weak}}}

\usepackage{sectsty}
\usepackage[colorinlistoftodos,prependcaption,textsize=footnotesize]{todonotes}

\sectionfont{\large}
\subsectionfont{\normalsize}
\subsubsectionfont{\normalsize}
\paragraphfont{\normalsize}

\newcommand{\on}{\omega_{\nu}}

\newcommand{\eps}{\varepsilon}

\def\Xint#1{\mathchoice
{\XXint\displaystyle\textstyle{#1}}%
{\XXint\textstyle\scriptstyle{#1}}%
{\XXint\scriptstyle\scriptscriptstyle{#1}}%
{\XXint\scriptscriptstyle\scriptscriptstyle{#1}}%
\!\int}
\def\XXint#1#2#3{{\setbox0=\hbox{$#1{#2#3}{\int}$ }
\vcenter{\hbox{$#2#3$ }}\kern-.59\wd0}}
\def\avint{\Xint-}

\newtheorem{theorem}{Theorem}
\newtheorem{lemma}{Lemma}
\newtheorem{remark}{Remark}
\newtheorem{definition}{Definition}

\newcommand{\unuk}{u_{\nu_k}}
\pagecolor{white}

\begin{document}

\phantom{ }
\vspace{4em}

\begin{flushleft}
{\large \bf Renormalization and energy conservation for axisymmetric fluid flows}\\[2em]
{\normalsize \bf Camilla Nobili$^1$ and Christian Seis$^2$}\\[0.5em]
\small \begin{tabbing} 
$^1$ \= Fachbereich Mathematik, Universit\"at Hamburg, Germany.\\
\>E-mail: camilla.nobili@uni-hamburg.de\\
$^2$ \> Institut f\"ur Analysis und Numerik,  Westf\"alische Wilhelms-Universit\"at M\"unster, Germany.\\
\>E-mail: seis@wwu.de\\[1em]
\end{tabbing}
{\bf Abstract:} We study vanishing viscosity solutions to the axisymmetric Euler equations with (relative) vorticity in $L^p$ with $p>1$. We show that these solutions satisfy the corresponding vorticity equations in the sense of renormalized solutions. Moreover, we show that the kinetic energy is preserved provided that $p>3/2$ and  the vorticity is nonnegative and has finite second moments.

\end{flushleft}

\tableofcontents

\vspace{2em}

\section{Introduction}
For axisymmetric incompressible flows without swirl, the (originally three-di\-men\-sion\-al) Navier--Stokes and Euler equations can be reduced to two-dimensional mathematical models which are obtained by assuming a cylindrical symmetry for both the physical space variables and the velocity components. Despite this simplification, such flows are still able to describe interesting physical phenomena like the motion and interaction of toroidal vortex rings. On the mathematical level, even though two-dimensional, the (vaguely defined) degree of difficulty of analyzing solution properties lies somewhere between that of the two-dimensional planar equations and the full three-dimensional model. Indeed, as we shall see later on, axisymmetric flows\footnote{From here on we shall omit the specification \emph{without swirl} for convenience.} do still feature vortex stretching and some of the standard global estimates have an unambiguous three-dimensional character. On the other hand, many of the features of the Biot--Savart kernel are typically two-dimensional even though some helpful symmetry properties are lost. 

In the present work, our aim is to study renormalization and energy conservation of solutions to the Euler equations that are obtained as  vanishing viscosity solutions from the axisymmetric Navier--Stokes equations. Here, renormalization is to be understood in the sense of DiPerna and Lions \cite{DiPernaLions89}, that is, a solution is called renormalized if the chain rule of differentiation applies in a suitable way. 
We are particularly interested into solutions whose vorticity is merely $L^p$ integrable in a sense that will be made precise later.

The analogous  (though in some parts technically much simpler) studies for the two-dimensional planar equations have been conducted quite recently: As long as the vorticity is $L^p$ integrable with exponent $p\ge 2$, DiPerna and Lions's theory for transport equations (combined with Calder\'on--Zygmund theory) ensures that the vorticity is a renormalized solution of the corresponding vorticity equation \cite{FilhoMazzucatoNussenzveig06}. This fact is true regardless of the construction of the solution. If $p\in (1,2)$, renormalization properties are proved in \cite{CrippaSpirito15} for vanishing viscosity solutions. The argument in this work relies on a duality argument and exploits the DiPerna--Lions theory. This theory, however, does not apply to the $p=1$ case, in which the associated velocity gradient is a singular integral of an $L^1$ function. Instead, a stability-based theory for continuity equations proposed in \cite{Seis17,Seis18} can be suitably generalized in order to handle this  situation and  to extend the results from  \cite{FilhoMazzucatoNussenzveig06,CrippaSpirito15} to the limiting case $p=1$; see \cite{CrippaNobiliSeisSpirito17}.

Conservation of kinetic energy for vanishing viscosity solutions with $L^p$ vorticity, $p>1$, is established in \cite{CheskidovFilhoLopesShvydkoy16} for the planar two-dimensional setting (on the torus). The corresponding three-dimensional problem  gained much attention in recent years, particularly in connection with Onsager's conjecture \cite{Onsager49}, which states that the threshold H\"older regularity for the validity of energy conservation is the exponent $1/3$. Energy conservation for larger H\"older exponents was proved in 
\cite{ConstantinETiti94}, see also \cite{Eyink94} for partial results and \cite{CheskidovConstantinFriedlanderShvydkoy08} for improvements, while the sharpness of this exponent was proved in \cite{Isett18}, building up on the theory developed in \cite{DeLellisSzekelyhidi09,DeLellisSzekelyhidi13,BuckmasterDeLellisIsettSzekelyhidi15,Buckmaster15,BuckmasterDeLellisSzekelyhidi16}. 
We note that H\"older-$1/3$ regularity is guaranteed for any vorticity field in $L^p$ with $p \ge \frac92$ thanks to three-dimensional Sobolev embeddings.

Before discussing our precise findings and the relevant earlier results  for the axisymmetric  equations, we shall introduce the mathematical model. The Euler equations for an  ideal fluid in $\R^3$ are given by the system
\begin{align}
\partial_t u + u\cdot \grad_x u + \grad_x p & =0,\label{Euler1}\\
\grad_x\cdot u &=0,\label{Euler2}
\end{align}
where $u = u(t,x)\in \R^3$ is the fluid velocity and $p=p(t,x)\in \R$ is the pressure. In this formulation, the (constant) fluid density is set to $1$. Whenever the fluid has locally finite kinetic energy, which will be the case in the regularity framework considered in this paper, the Euler equations can be interpreted in the sense of distributions. 

\begin{definition}\label{def:Euler}
Let $T>0$ and $u_0\in L^2_{\loc}(\R^3)^3$ be given. A vector field $u\in L^2_{\loc}((0,T)\times\R^3)^3$ is called a distributional solution to the Euler equations \eqref{Euler1}, \eqref{Euler2} if 
\[
\int_0^T \int_{\R^3} \left(\partial_t F\cdot u +\grad_x F:u\otimes u\right)\, dxdt + \int_{\R^3} F(t=0)\cdot u_0\, dx = 0
\]
for any divergence-free vector field $F\in C_c^{\infty}([0,T)\times \R^3)^3$ and 
\[
\int_0^T \int_{\R^3} \grad_x f\cdot u\, dxdt = 0
\]
for any   $f\in C_c^{\infty}([0,T)\times \R^3)$. 
\end{definition}

We  restrict ourself to the case of axisymmetric solutions without swirl. That is,  if $(r,\theta,z)$ are the cylindrical coordinates of a point $x\in \R^3$, i.e., $x= (r\cos\theta,r\sin\theta,z)^T$, we shall assume that 
\[
u = u(t,r,z),\quad\mbox{and}\quad  u = u^re_r + u^ze_z,
\]
where $e_r$ and $e_z$ are the unit vectors in radial and vertical directions, which form together with the angular unit vector $e_{\theta}$ a basis of $\R^3$, 
\[
e_r=(\cos\theta,\sin\theta,0)^T,\quad e_{\theta}=(-\sin\theta,\cos\theta,0)^T,\quad e_z=(0,0,1)^T.
\]
We remark that $u^\theta = u\cdot e_{\theta}$ is  the swirl direction of the flow, that we assume to vanish identically.
Under these hypotheses on the velocity field, the vorticity vector is unidirectional, $\grad_x\times u = (\partial_z u^r - \partial_r u^z)e_{\theta}$, and we write $\omega = \partial_z u^r - \partial_r u^z$. A direct computation reveals that this quantity, that we will call vorticity from here on, satisfies the continuity equation
\begin{equation}\label{omega}
\partial_t \omega + \partial_r (u^r \omega) + \partial_z(u^z\omega) = 0
\end{equation}
on the half-space $\H = \left\{(r,z)\in \R^2: r>0\right\}$. We remark that $\omega$ is thus a conserved quantity, because the no-penetration boundary condition $u^r=0$ on $\partial \H$ comes along with the symmetry assumptions. However, opposed to the situation for the two-dimensional planar Euler equations, the vorticity is not transported by the flow, as the divergence-free condition \eqref{Euler2} becomes 
\begin{equation}\label{div}
r^{-1}\partial_r(ru^r)+\partial_zu^z=0
\end{equation}
in cylindrical coordinates. Indeed, the continuity equation can be rewritten as a damped transport equation,
\[
\partial_t \omega + u^r\partial_r \omega + u^z\partial_z \omega = \frac1r u^r \omega,
\]
where the   damping term on the right-hand side describes the phenomenon of vortex stretching, $\frac1r u^r \omega e_{\theta}= (\grad_x\times u)\cdot \grad_x u$. What is transported instead is the relative vorticity $\xi = \omega/r$,  
\begin{equation}\label{xi}
\partial_t \xi + u^r\partial_r \xi + u^z\partial_z \xi = 0.
\end{equation}

We remark that the flow is entirely determined by the (relative) vorticity, as the associated velocity field can be reconstructed with the help of the Biot--Savart law in $\R^3$,
\begin{equation}\label{BS}
u(t,x) = \frac1{4\pi} \int_{\R^3} \frac{x-y}{|x-y|^3}\times e_{\theta}(y) \omega(t,y)\, dy.
\end{equation}
A  transformation into cylindrical coordinates and an analysis of the axisymmetric Biot--Savart law can be found, for instance,  in \cite{GallaySverak15}.
%
%
 
 Thanks to this relation, we may thus study \eqref{xi}, \eqref{BS} instead of  \eqref{Euler1}, \eqref{Euler2}. Working with the vorticity formulation has certain advantages: At least on a formal level, it is readily seen that the vorticity equation \eqref{xi} preserves any $L^p$ norm,
  \begin{equation}\label{vorticity}
 \|\xi(t)\|_{L^p(\R^3)} = \|\xi_0\|_{L^p(\R^3)}\quad \forall t \ge 0,
 \end{equation}
 if $\xi_0$ is the initial relative vorticity.\footnote{
We caution the reader that throughout the manuscript, we carefully distinguish between the Lebesgue spaces on the full three-dimensional space, $L^p(\R^3)$, and those on the two-dimensional half-space $L^p(\H)$.  Notice also that    the three-dimensional Lebesgue measure reduces to the weighted measure $ 2\pi r d(r,z)$ on  $\H$ when restricted to axisymmetric configurations as in \eqref{vorticity}. In particular, $\|\xi\|_{L^1(\R^3)} = 2\pi \|\omega\|_{L^1(\H)}$.}
 This observation is crucial, for instance, in order to prove uniqueness in the case of bounded vorticity fields \cite{Danchin07}. 
 The drawback of working with \eqref{xi} is that there is no direct way of giving a meaning to the transport term in low integrability settings (opposed to the momentum equation \eqref{Euler1}).  For instance, it is not obvious to us, how to extend  common symmetrization techniques that allow for an alternative formulation of the transport nonlinearity in the planar two-dimensional setting, see, e.g., \cite{Delort91,VecchiWu93,BohunBouchutCrippa16}.  
   
 Whenever the product $u\xi$ is locally integrable, we can interpret the transport equation  \eqref{xi} in the sense of distributions.
 
 \begin{definition}\label{D1}
 Let $T>0$ and $p,q\in (1,\infty)$ be given with $\frac1p+\frac1q =1$. Let $\xi_0\in L^p_{\loc}(\H)$   and $u\in L^1((0,T);L^q_{\loc}(\H)^3)$ be such that $r^{-1}\partial_r(ru^r) +\partial_zu^z=0$. Then $\xi\in L^{\infty}((0,T);L^p_{\loc}(\H))$ is called a \emph{distributional solution} to the transport equation \eqref{xi} with initial datum $\xi_0$ if
 \[
 \int_0^T\int_{\H} \xi \left(\partial_t f + u^r\partial_rf+ u^z\partial_z f\right) r\, d(r,z)dt + \int_{\H} \xi_0 f(t=0)r\, d(r,z) = 0
 \]
 for any $f \in C_c^{\infty}([0,T)\times \H)$. 
 \end{definition}
 Notice that the definition  provides a distributional formulation of the continuity equation \eqref{omega} in which $\omega$ is replaced by $r\xi$. 
 
Simple scaling arguments show that the local integrability of the product $u\xi$ can be expected to hold true only if $p\ge 4/3$. For this insight, it is crucial to observe that the  Sobolev inequality
\begin{equation}\label{embedding}
\|u\|_{L^{\frac{2p}{2-p}}(\H)}\lesssim \|\omega\|_{L^p(\H)}
\end{equation}
is valid  as in the planar two-dimensional setting, cf.~\cite[Proposition 2.3]{GallaySverak15}. For vorticity fields with smaller integrability exponents, we propose the notion of renormalized solutions.
 
  \begin{definition}\label{D2}
 Let $T>0$ be given. Let $\xi_0\in L^1_{\loc}(\H)$   and $u\in L^1((0,T);L^1_{\loc}(\H)^3)$ be such that $r^{-1}\partial_r(ru^r) +\partial_zu^z=0$. Then $\xi\in L^{\infty}((0,T);L^1_{\loc}(\H))$ is called a \emph{renormalized solution} to the transport equation \eqref{xi} with initial datum $\xi_0$ if
 \[
 \int_0^T\int_{\H} \beta(\xi)\left(\partial_t f + u^r\partial_rf+ u^z\partial_z f\right) r\, d(r,z)dt + \int_{\H} \beta(\xi_0) f(t=0)r\, d(r,z) = 0
 \]
 for any $f \in C_c^{\infty}([0,T)\times \H)$ and any bounded $\beta\in C^1(\R)$ vanishing near zero. 
 \end{definition}
We remark that the notion of renormalized solutions  implies the conservation of the $L^p$ integral of vorticity in the sense of \eqref{vorticity} via a standard approximation argument. Moreover, it is shown in \cite{DiPernaLions89,Ambrosio04} that renormalized solutions are transported by the Lagrangian flow of the vector field $u$ as in the smooth situation. We will further comment on this in Section \ref{sec:DP} below. The relation between Lagrangian transport and the partial differential equations \eqref{omega} and \eqref{xi} was thoroughly reviewed in \cite{AmbrosioCrippa14}. 
 
 In the present paper, we study solutions to the vorticity equation \eqref{xi} in the case where the initial (relative) vorticity is unbounded, more precisely,
 \begin{equation}\label{initial}
\xi_0\in L^1\cap L^p(\R^3)
\end{equation}
for some $p\in (1,\infty)$. We are thus outside of the class of functions in which uniqueness is known to hold \cite{Danchin07,AbidiHmidiKeraani10}. On the positive side, existence of  distributional solutions to the Euler equations \eqref{Euler1}, \eqref{Euler2} was proved in \cite{JiuWuYang15} for initial vorticities satisfying \eqref{initial} and under the additional assumption that the  initial kinetic energy is finite,  $u_0\in L^2(\R^3)^3$. (Notice that local $L^2$ bounds on the initial velocity can be deduced from the integrability assumptions on the vorticity via Sobolev embeddings, cf.~\eqref{embedding}.) For larger integrability exponents and (near) vortex sheet initial data, (crucial insights on) existence results were previously obtained in \cite{UkhovskiiIudovich68,Delort92,SaintRaymond94,ChaeKim97,ChaeImanuvilov98,JiuXin04,JiuXin06}. To the best of our knowledge,  renormalized solutions (Definition \ref{D2}) have  not been considered  in the context of the axisymmetric Euler equations.
 
 We are particularly interested into solutions that are obtained as the vanishing viscosity limit from the Navier--Stokes equations, which are, in fact, physically meaningful approximations to the Euler equation. Hence, for any viscosity constant $\nu>0$, we consider solutions $(\unu,p_{\nu})$ to the Navier--Stokes equations
 \begin{align}
 \partial_t \unu + \unu\cdot \grad_x \unu + \grad_x p_{\nu}  &= \nu\laplace_x \unu \label{NS1},\\
 \grad_x\cdot \unu &=0.\label{NS2}
 \end{align}
We furthermore impose fixed initial conditions, $\unu(0)=u_0$ and shall assume that $\unu$ is axisymmetric, that is, $\unu = \unu(t,r,z)$ and $\unu = (\unu)^re_r + (\unu)^ze_z$.

 Instead of working with the momentum equation \eqref{NS1}, will mostly study its vorticity formulation, which is a viscous version of \eqref{omega} (or \eqref{xi}),  see \eqref{omega-nu} (or \eqref{xinu}) below. It was shown in \cite{GallaySverak15} that under the assumption \eqref{initial} on the initial data, which implies that $\omega_0\in L^1(\H)$, there exists a unique global (mild) solution $\omega^{\nu}\in C^0([0,\infty)\times L^1(\H)) \cap C^0((0,\infty)\times L^{\infty}(\H))$ to the viscous vorticity equation.

 Starting from this solution to the Navier--Stokes equations, our first result addresses compactness and  convergence   to the Euler equations. 
 \begin{theorem}[Compactness and convergence to Euler]\label{thm:compactness}
Let $\unu$ be the unique solution to the Navier--Stokes equations \eqref{NS1}, \eqref{NS2}  with initial datum $u_0\in L^2_{\loc}(\R^3)$ such that the associated relative vorticity $\xi_0$ belongs to $L^1 \cap L^p(\R^3)$ for some $p>1$.  Then there exist $u \in C([0,T];L^2_{\loc}(\R^3)^3)$  with $\grad u\in L^{\infty}((0,T);L^p_{\loc}(\R^3)^{3\times 3})$ and $\xi \in L^{\infty}((0,T);L^1\cap L^p(\R^3))$ and a subsequence $\{\nu_k\}_{k=0}^{\infty}$ such that 
\[
 u_{\nu_k}\to u \quad \mbox{strongly in }C([0,T];L^2_{\loc}(\R^3))
\]
and 
\[
\xi_{\nu_k}\to \xi  \quad \mbox{weakly-$\star$ in }L^{\infty}((0,T);L^p(\R^3)).
\]
Moreover, $u$ is a distributional solution to the Euler equations \eqref{Euler1}, \eqref{Euler2} and $\omega = r\xi $ is the corresponding  vorticity that is (in a distributional sense) related to $u$ by the Biot--Savart law \eqref{BS}.
\end{theorem}
 The vanishing viscosity limit was studied for finite energy solutions with mollified initial datum satisfying the bound \eqref{initial} in \cite{JiuWuYang15}. The novelty in the above result is the kinetic energy may be unbounded. For earlier and related convergence results for non-classical solutions, we refer to \cite{UkhovskiiIudovich68,JiuXin04,HmidiZerguine09,Wu10,Abe18} and references therein.
 
 Our next statement concerns the renormalization property of the relative vorticity.

 \begin{theorem}[Renormalization]\label{thm:renorm}
Let $u$ and $\xi$ be  the velocity field and relative vorticity, respectively, from  Theorem \ref{thm:compactness}. Then $\xi$ is a renormalized solution to the transport equation \eqref{xi} with velocity $u$. In particular, it holds that
 \[
 \|\xi(t)\|_{L^p(\R^3)}  = \|\xi_0\|_{L^p(\R^3)}
 \]
 and $\xi$ is transported by the regular Lagrangian flow of $u$ in $\R^3$.
 \end{theorem}
 
 To the best of our knowledge, in this result, renormalized solutions to the axisymmetric Euler equations are considered for the first time. We recall from the above discussion that for $p\in (1,4/3)$, the interpretation of the transport equation \eqref{xi} as a distributional solution does not apply anymore as the  transport nonlinearity is no longer integrable. In particular, while for $p\ge 4/3$ our result implies that distributional and renormalized solutions coincide, in the low integrability range, we show the existence of renormalized solutions. We also recall that for $p\ge 2$, the result in Theorem \ref{thm:renorm} is already covered in DiPerna and Lions's original paper \cite{DiPernaLions89}. In section \ref{sec:DP}, we recall the theory from \cite{DiPernaLions89} and explain what we mean by $\xi$ being transported by a flow. For a precise definition of regular Lagrangian flows, we refer to \cite{Ambrosio04,AmbrosioCrippa14}. 
 
Our final result addresses the conservation of the kinetic energy.

 \begin{theorem}\label{thm:energy}
 Let $p\ge \frac32$. Suppose that the fluid has finite  kinetic energy, $u_0\in L^2(\R^3)^3$, and that  $\omega_0$ is nonnegative and has finite impulse,
 \[
 \int_{\H} \omega_0 r^2\, d(r,z) <\infty.
 \]
 Then the kinetic energy is preserved,
 \[
 \|u(t)\|_{L^2(\R^3)} = \|u_0\|_{L^2(\R^3)}.
 \]
 \end{theorem}
In order to show conservation of energy, the growth of vorticity at infinity has to be suitably controlled. Here, we choose a growth condition that is natural as it can be interpreted as the control of the fluid impulse. Notice that the latter is conserved by the evolution, cf.~Lemma \ref{L2}. This is in principle not required by our method of proving Theorem \ref{thm:energy}, and any estimate of the form $\|r^2\omega(t)\|_{L^1(\H)}\lesssim \|r^2\omega_0\|_{L^1(\H)}$ would be sufficient. It is, however, not clear to us whether such an estimate holds true under our integrability assumptions apart from the special case considered in Lemma \ref{L2}, that is, for nonnegative (or nonpositive) vorticity fields. Also, if higher order moments could be controlled, our method shows that the value of $p$ could be lowered (at least up to $p>\frac65$).  See, for instance,  \cite{ChaeKim97} for similar results in the setting with $p>3$ (and general solutions).
 We leave this issue for future research and consider the simplest case here.

From the result in Theorem \ref{thm:energy}, it follows that were are outside of the range in which Kolmogorov's celebrated K41 theory of three-dimensional turbulence applies, since, similar to the case of planar two-dimensional turbulence, there cannot be  anomalous diffusion.

From here on, we will simplify the notation by writing $\grad = {\partial_r \choose \partial_z}$, with the interpretation that $\grad\cdot f  = \partial_r f^r + \partial_z f^z$ while $\grad_x\cdot f = \partial_1 f^1 + \partial_2f^2 +\partial_3 f^3$ is the divergence with respect to a Cartesian basis. The advective derivatives $f\cdot \grad$ and $f\cdot \grad_x$ are to be interpreted correspondingly.

The remainder of the article is organized as follows: In Section \ref{sec:DP} we recall the parts of the DiPerna--Lions theory for transport equations and explain how the results apply to the setting under consideration. In Section \ref{sec:est} we provide estimates for the velocity field that are essentially based on the Biot--Savart law. Section \ref{sec:glob} contains global estimates for the axisymmetric Navier--Stokes equations, while the proof of Theorems \ref{thm:compactness}, \ref{thm:renorm} and \ref{thm:energy} are given in Sections \ref{RS+VV}, \ref{sec:renorm} and \ref{sec:energy}, respectively. This work, finally, contains an appendix in which a helpful interpolation estimate is provided.

 \section{Renormalized solutions for linear transport equations}\label{sec:DP}

In this section, we shall briefly recall DiPerna and Lions's theory for linear transport equations \cite{DiPernaLions89} and explain how it applies to the situation at hand. We are particularly interested into well-posedness and renormalization properties  of the transport equation \eqref{xi}, which we shall now treat as a (linear) passive scalar equation
\begin{equation}\label{23}
\partial_t \theta + u\cdot \grad \theta = 0
\end{equation}
for some scalar quantity $\theta$ and a velocity field $u$  that  does not depend on $\theta$. Yet, we have in mind that $u$ has its origin in the fluid dynamics problem that is considered in the main part of this paper. We shall thus continue assuming that the flow is incompressible in the sense of \eqref{div},
and that $u^r=0$ on $\partial \H$, which is ensured in the nonlinear setting by the Biot--Savart law \eqref{BS}, see also the discussion in \cite{GallaySverak15}. 
Working in cylindrical coordinates becomes at this point problematic as the cylindrical divergence of the velocity field $u$ might in general be unbounded opposed to the Cartesian divergence, which vanishes identically.  In order to apply the DiPerna--Lions theory, in which that boundedness is a crucial assumption, it is therefore advantageous  to go back to the Cartesian formulation and rewrite \eqref{23} as
\[
\partial_t \theta + u\cdot \grad_x \theta = 0.
\]
If, in addition, $u$ is Sobolev regular, as is the case for the axisymmetric Euler equations under the integrability assumption \eqref{vorticity} on the vorticity, the theory in \cite{DiPernaLions89} applies. We summarize some of the main results, again formulated for the axisymmetric setting, and not aiming for the most general assumptions.


%

\begin{theorem}[\cite{DiPernaLions89}]\label{thm:DP}
 Let $T>0$ and $p\in (1,\infty)$ be given and  $\theta_0\in L^p(\R^3)$   and $u\in L^1((0,T);W^{1,1}_{\loc}(\H)^3)$ be such that $r^{-1}\partial_r(ru^r) +\partial_zu^z=0$ and
 \begin{equation}\label{growth}
 \frac{|u|}{1+r+|z|} \in L^1((0,T)\times \R^3) + L^{\infty}((0,T)\times \R^3).
 \end{equation}

\noindent
(i) There exists a unique renormalized  solution $\theta\in L^{\infty}((0,T);L^p(\R^3))$ of the transport equation \eqref{23} with initial datum $\theta_0$. 

\noindent
(ii) This solution is stable under approximation in the following sense: Let $\{\theta_0^k\}_{k\in \N}$ be a sequence that approximates $\theta_0$ in $L^p(\R^3)$ and $\{u^k\}_{k\in\N}$ a sequence that approximates $u$ in $ L^1((0,T);W^{1,1}_{\loc}(\H)^3)$ and such that $r^{-1}\partial_r(ru_k^r) +\partial_zu_k^z=0$. Let $\theta^k$ denote the corresponding renormalized solution. Then $\theta^k\to \theta$ strongly in $C([0,T];L^p(\R^3))$.

\noindent
(iii) If $q\in (1,\infty)$ is such that $\frac1p+\frac1q\le1$ and $u\in L^1((0,T);W^{1,q}_{\loc}(\H)^3)$, then distributional solutions are renormalized solutions and vice versa.
\end{theorem}

It has been proved in \cite{DiPernaLions89,Ambrosio04} that renormalized solutions are in fact transported by the (regular) Lagrangian flow of the vector field $u$, and this feature carries over to the cylindrical setting. Hence, it holds that $\theta(t,\phi(t,r,z)) = \theta_0(r,z)$, where $\phi$ satisfies a suitably generalized formulation of the ordinary differential equation
\[
\partial_t \phi(t,r,z) = u(t,\phi(t,r,z)),\quad \phi(0,r,z)=(r,z).
\]
In terms of the vorticity, the transport identity can be rewritten as $\omega(t,\phi(t,r,z)) =  \omega_0(r,z)\phi^r(t,r,z)/r $, and thus, $r/\phi^r(t,r,z)$ is the Jacobian.  See also \cite{AmbrosioCrippa14} for a review of the connection between the Lagrangian and Eulerian descriptions of transport by non-smooth velocity fields.

It follows a discussion of the validity of the growth condition \eqref{growth} in the Euler setting. 

\begin{remark}\label{rem}
In the nonlinear setting in which $u$ can be reconstructed from $\theta = \xi$  with the help of the Biot--Savart kernel \eqref{BS}, the growth condition \eqref{growth} it automatically fullfilled provided that $\xi\in L^{\infty}((0,T);L^1(\R^3))$ as it is assumed in this paper. Indeed, it is proved in \cite{GallaySverak15} that the axisymmetric Biot--Savart kernel satisfies similar decay estimates as the planar two-dimensional one, namely, if $G$ is obtained from restricting the three-dimensional Biot--Savart kernel to the axisymmetric setting, so that 
\[
u(r,z) = \int_{\H} G(r,z,\bar r,\bar z) \omega(\bar r,\bar z)\, d(\bar r,\bar z),
\]
it holds that
\[
|G(r,z,\bar r,\bar z)| \lesssim \frac1{|r-\bar r| + |z-\bar z|},
\]
cf.~\cite[Eq.~(2.11)]{GallaySverak15}. We now denote by $G_1$ the restriction of $G$ to the unit ball $B_1(r,z)$ and set $G_2=G-G_1$, and decompose  $u = u_1+ u_2$ accordingly.
Then, on the one hand, by Young's convolution inequality,
\begin{align*}
 \|u_1\|_{L^1(\H)} \lesssim \|(\chi_{B_1(0)}\frac1{|\cdot|})\ast |\omega| \|_{L^1(\H)}  \le \|\chi_{B_1(0)}\frac1{|\cdot|}\|_{L^1(\H)}\|\omega\|_{L^1(\H)} \lesssim \|\omega\|_{L^1(\H)}.
\end{align*}
On the other hand, 
\[
  \|u_2\|_{L^{\infty}(\R^3)} \lesssim \|(\chi_{B_1(0)^c}\frac1{|\cdot|})\ast |\omega| \|_{L^1(\H)} \le \|\omega\|_{L^1(\H)} .
\]
Using that $\|\omega\|_{L^1(\H)} = \frac1{2\pi}\|\xi\|_{L^1(\R^3)}$, we  deduce \eqref{growth}.
\end{remark}

Following \cite{CrippaSpirito15,CrippaNobiliSeisSpirito17}, our strategy for proving that vanishing viscosity solutions to the axisymmetric Euler equations are renormalized solutions relies on  duality arguments both in the viscous and in the inviscid setting. In the latter, we quote a suitable duality theorem from DiPerna and Lions's original work.

 \begin{lemma}[\cite{DiPernaLions89}]\label{s1-L1}
Let $p,q\in (1,\infty)$ be given such that $\frac1p+\frac1q=1$.  Let $u$  satisfy the general assumptions of Theorem \ref{thm:DP} and let $\theta\in L^{\infty}((0,T);L^p(\R^3))$ be the renormalized solution to the transport equation \eqref{23} with initial datum $\theta_0\in L^p(\R^3)$.  
Let $\chi \in L^{1}((0,T);L^q(\R^3))$ be given and let $f\in L^{\infty}((0,T);L^q(\R^3))$  be a renormalized solution of the backwards transport equation  
\begin{equation}\label{tvp}
-\partial_tf-u\cdot \nabla f=\chi.
    \end{equation}
Then it holds
 \begin{align*}
 \MoveEqLeft \int_0^T \int_{\H}\theta\,\chi\, r d(r,z)dt\\
 & =\int_{\H} \theta(T, r,z) f(T,r,z)\, rd (r,z)-\int_{\H}  \theta(0,r,z)f(0,r,z)\, rd(r,z).
 \end{align*}
 \end{lemma}

\section{Estimates on the velocity field}\label{sec:est}

In this section, we provide some estimates on the velocity field that turn out to be helpful in the subsequent analysis. We continue denoting by $\omega$ and $\xi$ the vorticity and relative vorticity, respectively, of a given (steady) axisymmetric velocity field $u$, that is,  $\omega = \partial_zu^r-\partial_r u^z$ and $\xi =\omega/r$ independently from the Euler or Navier--Stokes  background. In particular, any of the following estimates are consequences of the explicit definitions or follow from suitable properties of the Biot--Savart kernel.

Our first result is a fairly standard identity for the enstrophy, that is, the (square of the) $L^2$ norm of the velocity gradient. 

\begin{lemma}\label{L3}
It holds that 
\[
\|\grad_x u\|_{L^2(\R^3)}=\|\omega \|_{L^2(\R^3)}.
\]
\end{lemma}

We provide the argument for this standard identity for the convenience of the reader.
\begin{proof} From the definition of the vorticity, we infer that
 \begin{eqnarray*}
 \frac1{2\pi} \|\omega\|_{L^2(\R^3)}^2
  &=&\int_{\H}(\partial_zu^r-\partial_r u^z)^2 r\, d(r,z)\\
  &=&\int_{\H}(\partial_zu^r)^2\,r\, d(r,z)+\int_{\H}(\partial_r u^z)^2\,r\, d(r,z)-2\int_{\H} \partial_zu^r\partial_r u^z\, r\, d(r,z).
 \end{eqnarray*}
 We have to identify  the third term on the right-hand side: It holds that
 \[
  -2\int_{\H} \partial_zu^r\partial_r u^z\, r\, d(r,z) = \int_{\H}(\partial_ru^r)^2\, r\, d(r,z) +\int_{\H}\frac{(u^r)^2}{r}\, d(r,z)  + \int_{\H}(\partial_zu^z)^2\, r\, d(r,z) .
  \]
  Indeed, using the no-penetration boundary condition $u^r=0$ on $\partial \H$ together with the incompressibility condition \eqref{div}, a multiple integration by parts reveals  on the one hand that
 \begin{eqnarray*}
 \int_{\H}\partial_zu^r\partial_r u^z\, r\, d(r,z)
 &=&-\int_{\H}u^r\partial_z\partial_r u^z\, r\, d(r,z)\\
 &=&-\int_{\H}u^r\partial_r(-\partial_r u^r-\frac{1}{r}u^{r})\, r\, d(r,z)\\
 &=&-\int_{\H}(\partial_ru^r)^2\, r\, d(r,z)-\int\frac{(u^r)^2}{r}\, d(r,z)\,.
 \end{eqnarray*}
On the other hand, it holds that
\begin{eqnarray*}
\int_{\H}\partial_zu^r\partial_r u^z\, r\, d(r,z)
&=&-\int_{\H}\partial_r (\partial_zu^r\, r)u^z\, d(r,z)\\
&=&-\int_{\H}\partial_r\partial_zu^ru^z \,r\, d(r,z)-\int_{\H}\partial_zu^ru^z\, d(r,z)\\
&=&-\int_{\H}\partial_z(-\partial_z u^z-\frac{1}{r}u^r)u^z \,r\, d(r,z)-\int_{\H}\partial_zu^ru^z\, d(r,z)\\
&=&-\int_{\H}(\partial_zu^z)^2\, r\, d(r,z)\,.
\end{eqnarray*}
It remains to notice that 
\begin{equation}\label{gradientu}
|\grad_x u|^2  = (\partial_r u^r)^2 + \frac1{r^2} (u^r)^2+ (\partial_z u^r)^2 +(\partial_r u^z)^2 + (\partial_z u^z)^2
\end{equation}
to conclude the statement of the lemma.
%
\end{proof}

In the following lemma, we provide a maximal regularity estimate for the velocity gradient in terms of the relative vorticity. Our proof relies on the classical theories by Calder\'on, Zygmund and Muckenhoupt. 

\begin{lemma}\label{weighted-MR}
 For $p\in (1,2)$ it holds that 
 \begin{equation}\label{pmaxreg}
 \|\frac1r\grad_x u\|_{L^p(\R^3)} \lesssim \|\xi\|_{L^p(\R^3)}.
 \end{equation}
\end{lemma}

\begin{proof}We note that in view of the Biot--Savart law \eqref{BS}, the velocity gradient can be represented as a singular integral of convolution type, $\grad_x u = K\ast (\omega e_{\theta})$, where $|K(x)|\sim \frac1{|x|^3}$. It is well-known that  Calder\'on--Zygmund theory  guarantees that 
\[
\|\grad_x u \|_{L^p(\R^3)} \lesssim \|\omega\|_{L^p(\R^3)}
\]
for any $p\in (1,\infty)$. Our goal is to produce a weighted version of this estimate, namely
\[
 \int_{\R^3}|\nabla_x u|^p\,m\, dx\lesssim \int_{\R^3}|\omega|^p\, m\, dx
\]
with  $m = m(r) = \frac1{r^p}$ and $r=\sqrt{x_1^2+x_2^2}$, which is nothing but \eqref{pmaxreg}.  We are thus led to the theory of Muckenhoupt weights: If $p\in (1,\infty)$ and $m$ is in the class of Muckenhoupt weights  $A_p$  then the weighted-maximal regularity estimate \eqref{pmaxreg} holds. Here, $A_p$ is the set of nonnegative locally integrable weight functions satisfying  
\begin{equation}\label{Ap}
\left(\avint_B m(x)\, dx\right)\left(\avint_B m(x)^{-\frac qp}\, dx \right)^{\frac pq}\le C
\end{equation}
for a universal constant $C>0$ and all balls $B$ in $\R^3$, and $q\in (1,\infty)$ with $\frac1p+\frac1q=1$. This   well known result that can be found, for instance,  in the book of Stein \cite{Stein}.

We thus have to show that $m = m(r) = r^{-p}$ satisfies \eqref{Ap} for $p\in (1,2]$.  
 For this, consider a ball in $\R^3$ with radius $R$ and centered in a generic point $X = (X_1,X_2,X_3) \in \R^3$, i.e., $B = B_R(X)$. We denote by $d$ the distance of $X$ to the $z $-axis, that is, $d = \sqrt{X_1^2+X_2^2}$. 
 We split our argumentation into the  two cases when $d\geq 2R$ (far field) and $d<2R$ (near field).

Let us first consider the case where $d\geq 2R$. Notice that  we have $d-R\le \sqrt{x^1_2+x_2^2}\le d+R$ for any $x\in B$ by the triangle inequality, and thus
  $$\frac{1}{(x_1^2+x_2^2)^{\frac p2}}\leq\frac{1}{(d-R)^p}\quad\mbox{and}\quad
(x_1^2+x_2^2)^{\frac q2}\leq (d+R)^{q}\,.$$
For $m(x) = (x_1^2+x_2^2)^{-\frac{p}2}$, we now compute
\[
 \avint_B m(x)\, dx    \le \frac1{(d-R)^p},
\]
and
\[
\left( \avint_B m(x)^{-\frac{q}p}\, dx \right)^{\frac{p}q}  \le (d+R)^p.
\]
Making use of the fact that $ \frac{d+R}{d-R}\le 3$ for all  $d \ge 2R$, we deduce that
\[
\left( \avint_B m(x)\, dx\right)\left( \avint_B m(x)^{-\frac{q}p}\, dx \right)^{\frac{p}q} \le \left(\frac{d+R}{d-R}\right)^p \le 3^p.
\]

We now turn to the case where $d< 2R$. We first observe that  $ \sqrt{x_1^2+x_2^2}<d+R$ and $ |x_3-X_3|<R$ for all $x\in B$, and we may thus bound the integral over the ball by an integral over the cylinder. Making  relative transformations in cylindrical coordinates, we then have the estimates
\[
  \avint_B m(x)\, dx   \lesssim \frac1{R^2}  \int_0^{d+R} \frac1{r^{p-1}}\, dr \lesssim \frac{(d+R)^{2-p}}{R^2},
\]
provided that $p< 2$, and
\[
 \left(\avint_B m(x)^{-\frac qp}\, dx\right)^{\frac pq}  \lesssim \left( \frac1{R^2}\int_0^{d+R} r^{q+1}\, dr\right)^{\frac{p}q}\lesssim \left(\frac{(d+R)^{q+2}}{R^2}\right)^{\frac{p}q}.
\]
Taking the product and using that $\frac{d+R}{R}\le 3$ for all $d<2R$, we conclude that
\[
\left( \avint_B m(x)\, dx\right)  \left(\avint_B m(x)^{-\frac qp}\, dx\right)^{\frac pq} \lesssim \left(\frac{d+R}{R}\right)^{2p}\le 3^{2p}.
\]

Hence, in either cases, we proved \eqref{Ap} and, thus, the proof is over. 
\end{proof}

\section{Global estimates for the axisymmetric Navier--Stokes equations} \label{sec:glob}
In this section, we provide some global estimates for solutions to the Navier--Stokes equations that we turn out to be helpful later on. We start by rewriting the momentum equation \eqref{NS1} in terms of the vorticity $\omega_{\nu} = \partial_zu_{\nu}^r - \partial_r u_{\nu}^z$ and the  relative vorticity $\xi_{\nu}=\omega_{\nu}/r$. The evolution equation for the vorticity is given by
\begin{equation}\label{omega-nu}
\partial_t \omeganu + \div(\unu\omeganu) = \nu\left(\laplace \omeganu +\frac1r\partial_r \omeganu - \frac1{r^2}\omeganu\right),
\end{equation}
and is equipped with homogeneous Dirichlet conditions on the boundary of the half-space, i.e. $\omeganu=0$ on $\partial\H$. It follows that the vorticity equation is  conservative, as expected, because $r^{-1}\partial_r\omeganu - r^{-2}\omeganu = \partial_r(r^{-1}\omeganu)$. The relative vorticity satisfies the nonconservative equation
\begin{equation}\label{xinu}
\partial_t \xinu +\unu\cdot \grad\xinu = \nu\left(\laplace \xinu + \frac3r\partial_r \xinu\right)
\end{equation}
which is supplemented with homogeneous Neumann boundary conditions, $\partial_r \xinu=0$ on $\partial \H$. We will mostly work with the latter equation. For initial data $\xinu(0) = \xi_0$ in $L^1(\R^3)\cap L^p(\R^3)$, cf.~\eqref{initial}, well-posedness for either formulation can be inferred from the theory developed by Gallay and \v{S}ver\'{a}k \cite{GallaySverak15}. In the following, $\omeganu$ will always be the unique mild solution to the vorticity equation \eqref{omega-nu} in the class $C([0,T);L^1(\H))\cap C((0,T);L^{\infty}(\H))$ and $\xinu=\omeganu/r$. We start by recalling some useful properties which can be found in various references. Yet, we provide their short proofs for the convenience of the reader. Our first concern is an $L^p$ estimate.

\begin{lemma}\label{lem1}
It holds that
\begin{equation}\label{bound-on-xi}
\|\xinu\|_{L^{\infty}((0,T);L^p(\R^3))} \le \|\xi_0\|_{L^p(\R^3)}.
\end{equation}
\end{lemma}

\begin{proof}We can perform a quite formal computation as solutions can be assumed to be smooth by standard approximation procedures. 
A direct calculation yields
\begin{eqnarray*}
 \frac{d}{dt}\frac1p\int_{\H} |\xin|^p\, r\, d(r,z) &=& \nu \int_{\H} |\xin|^{p-2}\xin\Delta\xin \, rd(r,z) + 3 \nu \int_{\H} |\xin|^{p-2}\xin\partial_r \xin\, d(r,z) ,
 \end{eqnarray*}
 where we made use of the no-penetration boundary conditions on the velocity field $u$ to eliminate the advection term. The Cartesian Laplacian $\laplace_x = \laplace + \frac1r\partial_r$ is coercive, because
 \[
 \int_{\H} |\xin|^{p-2} \xin \left(\laplace \xin +\frac1r\xin\right)\, rd(r,z) = -(p-1) \int_{\H} |\xin|^{p-2} |\grad\xin|^2\, rd(r,z)\le 0
 \]
as can be seen by an integration by parts. Another integration by parts reveals that the first order term is nonpositive and can thus be dropped,
\[
\int_{\H} |\xin|^{p-2}\xin \partial_r \xin \, d(r,z) = \frac1p \int_{\H} \partial_r |\xin|^p\, d(r,z) = -\frac1p\int_{\partial \H} |\xin|^p \, d(r,z)\le 0.
\]
A combination of the previous estimates yields 
\begin{equation}\label{dissipation}
 \frac{d}{dt}\frac1p\int_{\H} |\xin|^p\, r\, d(r,z) +\nu (p-1) \int_{\H} |\xin|^{p-2} |\grad\xin|^2\, rd(r,z)\le 0,
\end{equation}
and an integration in time yields the desired estimate \eqref{bound-on-xi}.
\end{proof}

Our next estimate quantifies  integrability improving features of the advection-diffusion equation \eqref{xinu} by suitably extending the estimates on the $L^p$ norm established in the previous lemma to any $q\in [p,\infty)$.

\begin{lemma}\label{xi_improve}
For any $q\in [p,\infty]$, it holds that
 \begin{equation}\label{N10bis}
 \|\xin(t)\|_{L^q(\R^3)}\lesssim \left(\frac{1}{\nu t}\right)^{\frac 32(\frac{1}{p}-\frac1q)}\|\xi_0\|_{L^p(\R^3)}\quad\forall t>0.
 \end{equation}
\end{lemma}

\begin{proof} Our proof is a small modification of  the argument of Feng and \v{S}ver\'{a}k in \cite[Lemma 3.8]{FengSverak15}, where the case $p=1$ is considered. 
We define
 $E_q(t)=\|\xin(t)\|_{L^q(\R^3)}^q$
 for some $q\in [p,\infty)$ and claim that
 \begin{equation}\label{N10c}
 \frac{d}{dt}E_q(t)^{-\frac23} 
 \gtrsim  \nu\left(\int_{\R^3}|\xin|^{\frac{q}{2}}\, dx\right)^{-\frac 43}\,.
\end{equation}
Let us postpone the proof of this estimate a bit and explain first how it implies \eqref{N10bis}. Notice that, by interpolation of Lebesgue spaces, it is enough to show \eqref{N10bis} for exponents $q = 2^kp$ with $k\in \N_0$ and $q=\infty$.  We first treat the case for finite exponents, which will be achieved  by induction.  We start by observing that the 
 base case $k=0$ is settled in Lemma \ref{lem1} above.  The induction step from $k$ to $k+1$ is based on estimate \eqref{N10c}. We set $\tilde{q}=2^k$ and $q=2^{k+1}=2\tilde q$.
Plugging \eqref{N10bis} with $\tilde q = \frac{q}2$ into \eqref{N10c}, we find
\begin{eqnarray*}
 \frac{d}{dt}E_q(t)^{-\frac23}
 \gtrsim \nu(\nu t)^{2(\frac{\tilde q}{p}-1)}\|\xi_0\|_{L^p(\R^3)}^{-\frac{4}{3}\tilde q}
 = \nu^{\frac{q}{p}-1}t^{(\frac{q}{p}-2)}\|\xi_0\|_{L^p(\R^3)}^{-\frac{2}{3}q}.
\end{eqnarray*}
Integrating in time yields
\begin{eqnarray*}
 (E_q(t))^{-\frac{2}{3}}\geq(E_q(t))^{-\frac{2}{3}}-(E_q(0))^{-\frac{2}{3}}&\gtrsim& \nu^{\frac{q}{p}-1}\|\xi_0\|_{L^p(\R^3)}^{-\frac{2}{3}q}\int_0^t s^{(\frac{q}{p}-2)}\, ds\\
 &\sim& \nu^{\frac{q}{p}-1}\|\xi_0\|_{L^p(\R^3)}^{-\frac{2}{3}q} t^{(\frac{q}{p}-1)}\, 
\end{eqnarray*}
where we have used that $q>p$. Notice that all constants can be chosen uniformly in $q$. We have thus proved  \eqref{N10bis} for $q=2\tilde q$, which settles the case where $q= 2^kp$.

If $q=\infty$, we may now simply take the limit in \eqref{N10bis} and use the convegence of the Lebesgue norms, $\|\cdot\|_{L^{\infty}} = \lim_{q\to \infty} \|\cdot\|_{L^q}$.

It remains to provide the argument for \eqref{N10c}. We start by recalling that
 \begin{eqnarray*}
 -\frac{d}{dt}E_q(t)&\stackrel{\eqref{dissipation}}{\geq}&
 q (q-1) \nu \int_{\R^3}  |\xin|^{q-2} |\nabla \xin|^2 \, dx \sim \frac{q-1}q \nu \int_{\R^3} \left|\nabla |\xin|^{\frac{q}{2}}\right|^2 \, dx.
\end{eqnarray*}
Notice that the constants in the estimate can be chosen independently of $q$ as $q>1$, and can thus be dropped. We estimate the right-hand-side with the help  the 3D Nash inequality $\|f\|_{L^2(\R^3)} \lesssim \|f\|_{L^1(\R^3)}^{2/5}\|\grad f\|_{L^2(\R^3)}^{3/5}$,
and obtain
 \begin{eqnarray*}
 -\frac{d}{dt}E_q(t)  &\gtrsim& \nu\left(\int_{\R^3}|\xin|^{\frac{q}{2}}\, dx\right)^{-\frac 43}\left(\int_{\R^3}|\xin|^{q}\,dx\right)^{\frac{5}{3}},
\end{eqnarray*}
which can be rewritten as \eqref{N10c}.
\end{proof}

We also note that  the fluid impulse is conserved along the viscous flow. 
\begin{lemma}\label{L2}
Suppose that    $r^2 \omega_0 \in L^1(\H)$. Then
\[
\int_{\H} \on(t) r^2\, d(r,z) = \int_{\H} \omega_0 r^2 \, d(r,z).
\]
\end{lemma}

This identity can be seen in several ways, see, for instance \cite[Lemma 6.4]{GallaySverak15} for a proof that is based on the symmetry properties of the Biot--Savart kernel and applies to our regularity setting. We omit the proof and remark only that 
\[
\int_{\R^3} u\, dx = \frac12 \int_{\R^3} \omega e_{\theta}\times x\, dx = -\frac12 \left(\int_{\H} \omega r^2\, d(r,z)\right)e_z,
\]
whenever $u$ is an axisymmetric vector field and $\omega $ the associated vorticity. The conservation of momentum follows immediately from the Euler equations \eqref{Euler1}, \eqref{Euler2}.

The last global estimate concerns  the energy balance law, for which we assume that the initial kinetic energy is bounded.

\begin{lemma}\label{lem:balance}
Suppose that $\|u_0\|_{L^2(\R^3)} <\infty$. Then
\begin{equation}\label{N12}
\|\unu(t)\|_{L^2(\R^3)}^2 + \nu\int_0^t \|\grad_x \unu\|_{L^2(\R^3)}^2 \, dt = \|u_0\|_{L^2(\R^3)}^2\quad\mbox{for all }t>0.
\end{equation}
\end{lemma}

It is a classical result by Leray that for any divergence-free initial datum $u_0$ in $L^2(\R^3)$, there exists a weak solution to the Navier--Stokes equations \eqref{NS1}, \eqref{NS2} satisfying the energy inequality
\begin{equation}\label{NS3D-EI}
 \|\un(t)\|_{L^2(\R^3)}^2+\nu\int_0^t\|\nabla_x \un\|_{L^2(\R^3)}^2\, dt\leq   \|u_0\|_{L^2(\R^3)}^2,
\end{equation}
cf.~\cite{Leray34}. Whether there is an energy \emph{equality} \eqref{N12}   for such solutions is an important open problem.  There are various conditions available in the literature under which an equality can be established, most notably, Serrin's condition $u\in L^q((0,T);L^p(\R^d))$ with $\frac{d}p + \frac2q\le 1$   or Shinbrot's criterion $\frac2p+\frac2q\le 1$ and $p\ge 4$, cf.~\cite{Serrin63,Shinbrot74}. We refer to \cite{CheskidovLuo18} for an extension of the previous results to a larger class of function spaces and to \cite{BerselliChiodaroli18} for a recent improvement based on assumptions on the gradient of the velocity.

It is not difficult to see, that we can construct mild solutions in the setting of \cite{GallaySverak15} that satisfy the inequality \eqref{NS3D-EI}, and thus, thanks to the uniqueness in that setting,   our solutions do as well.  We remark that in \cite{BuckmasterVicol17} Buckmaster and Vicol construct  weak solutions for the three-dimensional Navier for which the energy inequality is not automatically achieved.
 Unfortunately, it is not obvious how to check  Serrin's or Shinbrot's integrability conditions to ensure an energy equality in the axisymmetric setting. The problem is the appearance of weights as, for instance, in \eqref{pmaxreg} and in suitable Sobolev inequalities. 
For this reason, we  provide a proof of \eqref{N12} that is tailored to our needs but still mimics the original arguments in \cite{Serrin63,Shinbrot74}.

\begin{proof}Thanks to the well-posedness result in \cite{GallaySverak15}, we may  suppose that \eqref{NS3D-EI} holds true in our setting. In particular,
we deduce
\begin{equation}\label{regularity1}
\un\in L^{\infty}((0,T); L^2(\R^3)^3)\quad \mbox{and} \quad \nabla_x \un\in L^2((0,T);L^2(\R^3)^{3\times 3})\,.
\end{equation}
In addition, thanks to the $L^p$ bound on the vorticity in Lemma \ref{lem1} and the weighted maximal regularity estimate in Lemma \ref{weighted-MR}, it holds that
\begin{equation}\label{regularity2}
\frac1r\grad_x \un\in L^{\infty}((0,T);L^p(\R^3)^{3\times 3}).
\end{equation}
By standard density arguments, we may thus find a sequence $\{\un^{\delta}\}_{\delta\downarrow 0}$ of divergence-free functions in $C_c^{\infty}((0,T);C_c^{\infty}(\R^3)^3)$ that converges towards $\un$ in $L^2((0,T); H^1_0(\R^3)^3)$ and stays bounded in all the spaces in which $\un$ is contained. We furthermore denote by $\eta^{\eps}$ a standard mollifier on $\R$. Because 
\[
F(t,x) =  \int_0^T\eta^{\eps}(t-\tau) \un^{\delta}(\tau,x)\, d\tau
\]
is an admissible test function in the definition of distributional solution of the Navier--Stokes equations, we find that
\begin{align*} 
\MoveEqLeft\int_0^T \int_{\R^3} \eta^{\eps}(T-\tau) \un^{\delta}(\tau,x) \cdot \un(T,x)\, dxd\tau\\
 & = \int_0^T \int_{\R^3} \eta^{\eps}(-\tau) \un^{\delta}(\tau,x)\cdot u_0(x)\, dxd\tau\\
&\quad+ \int_0^T \int_0^T\int_{\R^3} \frac{d \eta^{\eps}}{dt}(t-\tau)\un^{\delta}(\tau,x)\cdot \un(t,x) \, dxd\tau dt\\
&\quad -  \int_0^T \int_0^T\int_{\R^3} \eta^{\eps}(t-\tau)\grad_x\un^{\delta}(\tau,x) : \un(t,x)\otimes\un(t,x)\, dxd\tau dt\\
&\quad - \nu\int_0^T\int_0^T \int_{\R^3} \eta^{\eps}(t-\tau)\grad_x \un^{\delta}(\tau,x):\grad_x \un(t,x)\, dxd\tau dt.
\end{align*}
In a fist step, we send $\delta$ to zero with $\varepsilon>0$ fixed. The convergence is obvious for all but the nonlinear term. It is enough to show that the nonlinear term vanishes when $\un^{\delta}$ is replaced by $v^{\delta} = \un^{\delta} - \un$. Performing an integration by parts, we can throw the derivative on one of the $\un(t,x)$. H\"older's inequality then yields
\begin{align*}
\MoveEqLeft \left| \int_0^T \int_0^T\int_{\R^3} \eta^{\eps}(\tau-t)\grad_x v^{\delta}(\tau,x) : \un(t,x)\otimes\un(t,x)\, dxd\tau dt\right|\\
&\le \int_0^T \|\eta^{\eps}\ast v^{\delta}\|_{L^4(\R^3)} \|\un\|_{L^4(\R^3)} \|\grad_x \un\|_{L^2(\R^3)}\, dt,
\end{align*}
where by $\ast$ we denote the convolution-type operation between $\eta^{\eps}$ and $v^{\delta}$. We now have to make use of the interpolation inequality in Lemma \ref{weighted-interpolation} in the appendix and notice that $|\grad u|\le |\grad_x u|$ for any axisymmetric velocity field $u$. We find that
\begin{align*}
\MoveEqLeft \int_0^T \|\eta^{\eps}\ast v^{\delta}\|_{L^4(\R^3)} \|\un\|_{L^4(\R^3)} \|\grad_x \un\|_{L^2(\R^3)}\, dt\\
&\lesssim \int_0^T \|\eta^{\eps}\ast v^{\delta}\|_{L^2(\R^3)}^{\lambda} \|\eta^{\eps}\ast \grad_x v^{\delta}\|_{L^2(\R^3)}^{\frac12}\|\frac1r\eta^{\eps}\ast \grad_x v^{\delta}\|_{L^p(\R^3)}^{\frac12-\lambda}\\
&\quad\quad\quad \times \|\un\|_{L^2(\R^3)}^{\lambda}\|\grad_x \un\|_{L^2(\R^3)}^{\frac32}\|\frac1r\grad_x \un\|_{L^p(\R^3)}^{\frac12- \lambda}\, dt.
\end{align*}
Using H\"older's and Young's convolution inequality, we then infer that
\begin{align*}
\MoveEqLeft \int_0^T \|\eta^{\eps}\ast v^{\delta}\|_{L^4(\R^3)} \|\un\|_{L^4(\R^3)} \|\grad_x \un\|_{L^2(\R^3)}\, dt\\
&\lesssim \|  v^{\delta}\|_{L^{\infty}(L^2(\R^3))}^{\lambda} \|\frac1r \grad_x v^{\delta}\|_{L^{\infty}(L^p(\R^3))}^{\frac12-\lambda} \| \grad_x v^{\delta}\|_{L^2((0,T)\times \R^3)}^{\frac12}\\
&\quad\quad\quad \times \|\un\|_{L^{\infty}(L^2(\R^3))}^{\lambda}\|\grad_x \un\|_{L^2((0,T)\times \R^3)}^{\frac32}\|\frac1r\grad_x \un\|_{L^{\infty}(L^p(\R^3))}^{\frac12- \lambda}
\end{align*}
From \eqref{regularity1} and \eqref{regularity2} and the assumptions on $v^{\delta}$, we deduce that the right-hand side in the above estimate is vanishing as $\delta\to0$. Passing to the limit in the weak formulation of the Navier--Stokes equations above thus yields
\begin{align*} 
\MoveEqLeft\int_0^T \int_{\R^3} \eta^{\eps}(T-\tau) \un(\tau,x) \cdot \un(T,x)\, dxd\tau\\
 & = \int_0^T \int_{\R^3} \eta^{\eps}(-\tau) \un(\tau,x)\cdot u_0(x)\, dxd\tau\\
&\quad -  \int_0^T \int_0^T\int_{\R^3} \eta^{\eps}(t-\tau)\grad_x\un(\tau,x) : \un(t,x)\otimes\un(t,x)\, dxd\tau dt\\
&\quad - \nu\int_0^T\int_0^T \int_{\R^3} \eta^{\eps}(t-\tau)\grad_x \un(\tau,x):\grad_x \un(t,x)\, dxd\tau dt.
\end{align*}
Notice that the term that involved the time derivative on $\eta^{\eps}$ dropped out by imposing that $\eta^{\eps}$ is an even function.

We finally send $\eps$ to zero and may thus choose $\eps<T$ from here on. Notice first that
\[
\nu\int_0^T\int_0^T \int_{\R^3} \eta^{\eps}(t-\tau)\grad_x \un(\tau,x):\grad_x \un(t,x)\, dxd\tau dt\to\nu\int_0^T \|\grad_x\un\|_{L^2(\R^3)}^2\, dt
\]
thanks to standard convergence properties of the mollifier. For the convergence of the end-point integrals, we make use of the fact that our solutions are continuous in time with respect to the weak topology in $L^2(\R^3)$, see, e.g., \cite[Corollary 3.2]{Shinbrot74}. Because $\eta^{\eps}$ is chosen even, Lebesgue's convergence theorem then yields
\begin{align*}
\int_0^T \int_{\R^3} \eta^{\eps}(T-\tau) \un(\tau,x) \cdot \un(T,x)\, dxd\tau & \to \frac12\|\un(T)\|_{L^2(\R^3)}^2,\\
\int_0^T \int_{\R^3} \eta^{\eps}(-\tau) \un(\tau,x)\cdot u_0(x)\, dxd\tau& \to \frac12\|u_0\|_{L^2(\R^3)}^2.
\end{align*}

It remains to argue that the nonlinear term is vanishing. Notice first that
\[
\int_0^T \int_{\R^3} \un^{\delta} \cdot (\un\cdot \grad_x \un^{\delta})\, dxdt = \frac12 \int_0^T \int_{\R^3} \un\cdot \grad_x |\un^{\delta}|^2\, dxdt= 0
\]
for any $\delta$ if $\un^{\delta}$ is defined as above. This identity carries over to the limit $\delta\to0$ as can be seen by using the same kind of estimates that we used above in order to control the nonlinear term. We may thus rewrite the nonlinear term above as
\[
\int_0^T  \int_{\R^3}  (\un-\eta^{\eps}\ast \un) \cdot (\un\cdot \grad_x\un)\, dxdt,
\]
and, by applying the same kind of estimates again, we observe that this term vanishes as $\eps\to0$ by the convergence properties of the mollifier.
\end{proof}

\section{Vanishing viscosity limit. Proof of Theorem \ref{thm:compactness}}\label{RS+VV}

In this section, we turn to the proof of Theorem \ref{thm:compactness}. The compactness argument is based on the a priori estimate \eqref{bound-on-xi} on the relative vorticity and local estimates on the velocity field. The latter are provided by the following two lemmas.

\begin{lemma}\label{lem:local}
For any $R>0$ and any $p_*\in (1,p]\cap (1,2)$, there exists a constant $C(R)$ such that
\begin{align}
 \|\un\|_{L^{\infty}((0,T);W^{1,p_*}(B_R(0)))}& \le C(R)\left(\|\xi_0\|_{L^p(\R^3)}+\|\omega_0\|_{L^1(\H)}\right).\label{boundH1}
\end{align}
where $B_R(0)$ is the ball in $\R^3$ centered in the origin.
\end{lemma}

\begin{proof}
By standard interpolation between Lebesgue spaces, we may without loss of generality assume that $p = p_*$. The bound on the gradient is an immediate consequence of the maximal regularity estimate in Lemma \ref{weighted-MR} and formula \eqref{gradientu},
\[
\| \grad \un\|_{L^{\infty}((0,T);L^p(B_R ))}\le R^{p-1} \|r^{\frac1p-1}\grad \un\|_{L^{\infty}((0,T);L^p(\H))}\lesssim R^{p-1} \|\xi_0\|_{L^p(\R^3)},
\]
where $B_R = B_R^{\H}(0)$ denotes an open ball of radius $R$ centered at $0$ in the half-space $\H$. Notice that it is enough to show the statement of the lemma for $W^{1,p_*}(B_R)$ equipped with $d(r,z)$ instead of $W^{1,p_*}(B_R(0))$ equipped with $dx$.

In order to deduce an estimate on the velocity field itself, we first invoke the Poincar\'e estimate for mean-zero functions and the previous bound to observe that%
\begin{eqnarray}
\|  \un\|_{ L^p(B_R )}
 &\lesssim& R \| \grad \un\|_{ L^p(B_R)}+R^{\frac2p-2} \| \un\|_{ L^1(B_R )}\nonumber\\
 &\lesssim& R^p\|\xi_0\|_{L^p(\R^3)}+R^{\frac2p-2 }\| \un\|_{  L^1(B_R )}\label{Lp-estimate-u}
\end{eqnarray}
uniformly in time. It remains to bound the $L^1$ norm of $u$. For this purpose, we make use of the decay behavior of the Biot--Savart kernel. In \cite{GallaySverak15}, the authors show that the decay of the axisymmetric Biot--Savart kernel is identical (in scaling) to that of the planar Biot--Savart kernel, that is, if we rewrite \eqref{BS} as
\[
\un(r,z)  = \int_{\H} K(r,z,\bar r,\bar z)\on(\bar r,\bar z)\, d(\bar r,\bar z),
\]
then the kernel $K$ obeys the estimate
\[
|K(r,z,\bar r,\bar z)| \lesssim\frac1{|r-\bar r|+|z-\bar z|}.
\]
We thus have and write
\begin{align*}
 \int_{B_R}|\un(r,z)|\, d(r,z)
 &\lesssim\int_{B_R}\int_{\H}\frac{|\on(\bar r,\bar z)|}{|r-\bar r|+|z-\bar z|}\, d(\bar r,\bar z) d(r,z)\\
 &=\int_{B_R}\int_{B^{\H}_{2R}(r,z)}\frac{|\on(\bar r,\bar z)|}{|r-\bar r|+|z-\bar z| } \, d(\bar r,\bar z) d(r,z)\\
 &\quad +\int_{B_R}\int_{\H\setminus B^{\H}_{2R}(r,z)}\frac{|\on(\bar r,\bar z)|}{|r-\bar r|+|z-\bar z|} \, d(\bar r,\bar z)d(r,z).
\end{align*}
For the near-field, we use Fubini's theorem, Young's convolution estimate and Lemma \ref{lem1} to deduce
\begin{align*}
\int_{B_R}\int_{B^{\H}_{2R}(r,z)}\frac{|\on(\bar r,\bar z)|}{|r-\bar r|+|z-\bar z| } \, d(\bar r,\bar z) d(r,z)& \le  \int_{B_{3R}} \frac1{|r|+|z|} \, d(r,z)\int_{B_{2R}} |\on|\, d(r,z)\\
& \lesssim R\|\on\|_{L^1(\H)}  \le R \|\xi_0\|_{L^1(\R^3)}.
\end{align*}
For the far-field, we simply observe that the kernel is bounded below, and thus
\begin{align*}
\int_{B_R}\int_{\H\setminus B^{\H}_{2R}(r,z)}\frac{|\on(\bar r,\bar z)|}{|r-\bar r|+|z-\bar z|} \, d(\bar r,\bar z) &\lesssim  \frac1R\int_{B_R}\int_{\H\setminus B^{\H}_{2R}(r,z)}|\on(\bar r,\bar z) |\, d(\bar r,\bar z)\\
&\lesssim R\|\on\|_{L^1(\H)}  \le R \|\xi_0\|_{L^1(\R^3)}.
\end{align*}
Plugging the previous bounds into \eqref{Lp-estimate-u} yields \eqref{boundH1} as desired.
\end{proof}

\begin{lemma}\label{lem:dual}
For any $R>0$, it holds that
\[
\|\partial_t \un\|_{L^2((0,T);W^{-1,1}_{\sigma}(B_R(0))}\le C(R)\left(\|\xi_0\|_{L^p(\R^3)} + \|\omega_0\|_{L^1(\H)}\right),
\]
where $W^{-1,1}_{\sigma}(B_R(0))^3$ is the Banach space that is dual to the space of divergence-free vector fields in $W^{1,\infty}_0(B_R(0))^3$. 
\end{lemma}
The proof of this estimate is fairly standard. We sketch the argument for the convenience of the reader.

\begin{proof}
Let $F$ be a divergence-free vector field in $W^{1,\infty}_0(B_R(0))^3$. Then
\begin{align*}
\left(\un\cdot \grad \un, F\right)_{W^{-1,1}_\sigma(B_R(0))\times W^{1,\infty}_0(B_R(0))} &= - \int_{B_R(0)} \un\otimes \un: \grad F\, dx\\
&\le \|\un\|_{L^2(B_R(0))}^2 \| F\|_{W^{1,\infty}(B_R(0))},
\end{align*}
and a similar bound holds for the dissipation term $-\nu\laplace \un$. The statement thus follows directly from the momentum equation and Lemma \ref{lem:local}.
\end{proof}

We are now in the position to prove the compactness result. 

\begin{proof}[Proof of Theorem \ref{thm:compactness}] Thanks to Lemmas \ref{lem:local} and \ref{lem:dual}, the sequence of velocity fields $\{\un\}_{\nu\downarrow0}$ satisfies the hypotheses of the Aubin--Lions Lemma, and thus, for any $R>0$, there exists a subsequence that converges strongly in $C([0,T];L^2(B_R(0)))$. By applying a diagonal sequence argument, this convergence carries over to the space $C([0,T];L^2(K))$ for any compact $K$ in $\R^3$. Hence, there exists a subsequence (not relabelled) and a vector field $u\in C([0,T];L^2_{\loc}(\R^3)^3)$ such that 
\[
\un \to u\quad\mbox{strongly in }C([0,T];L^2_{\loc}(\R^3)).
\]
It is readily checked that $u$ is a distributional solution to the Euler equations \eqref{Euler1}, \eqref{Euler2}.

Moreover, from the a priori estimate on the relative vorticity in Lemma \ref{lem1}, we deduce that there exists a function $\xi \in L^{\infty}((0,T);L^p(\R^3))$ such that, upon taking a further subsequence, 
\[
\xi_n \to \xi \quad\mbox{weakly-$\star$ in }L^{\infty}((0,T);L^p(\R^3)).
\]
We finally notice that the velocity field $u$ and the vorticity $\omega  = r\xi$ are related by the Biot--Savart law that holds true in the sense of distributions.
\end{proof}

\section{Renormalization. Proof of Theorem \ref{thm:renorm}}\label{sec:renorm}
In this section, we provide the argument for the renormalization property of the relative vorticity obtained as the vanishing viscosity solution of the Navier--Stokes equations in Theorem \ref{thm:compactness}. Our approach is based on the duality formula in Lemma \ref{s1-L1} established in \cite{DiPernaLions89} and follows closely the argumentation from \cite{CrippaSpirito15,CrippaNobiliSeisSpirito17}.

In a first step, we show a compactness result for a backwards advection-diffusion equation, that is, as we will see, dual to the vorticity formulation \eqref{xinu} of the Navier--Stokes equations.

\begin{lemma}\label{lem:backwards}
Let $q\in (2,\infty)$   and $\chi\in L^1((0,T);L^q(\R^d))$ be given. Let $\fn$ denote the unique  solution in the class $ L^{\infty}((0,T);L^q(\R^3))$ with $\grad |\fn|^{\frac{q}2}\in L^2((0,T);L^2(\R^3))$    to the backwards advection-diffusion equation 
\[
-\partial_t \fn -\un\cdot \grad \fn  = \chi  + \nu\left(\laplace \fn -\frac1r\partial_r \fn\right)
\]
in $\H$ with finial datum $\fn(T)=0$ and homogeneous Dirichlet boundary conditions on $\partial \H$. Then 
 there exists a subsequence $\{\nu_k\}_{k\in\N}$ (the same as in Theorem \ref{thm:compactness}) such that 
\[
f_{\nu_k}\to f\quad\mbox{weakly-$\star$ in }L^{\infty}((0,T);L^q (\R^3)),
\]
where $f$ is the unique solution to the backwards transport equation \eqref{tvp}. 
\end{lemma}

We remark that renormalized solutions to advection-diffusion equations have been considered, for instance,  in \cite{DiPernaLions89,LeBrisLions04,Figalli08}.

\begin{proof}We start with an a priori estimate. A direct computation reveals that
\begin{align*}
\frac{d}{dt}\frac1q\int_{\R^3} |\fn|^q\, dx & = - \int_{\R^3} |\fn|^{q-2}\fn\chi\, dx +\nu (q-1) \int_{\R^3} |\fn|^{q-2} |\grad \fn|^2\, dx\\
&\ge  - \|\fn\|_{L^q(\R^3)}^{q-1} \|\chi\|_{L^q(\R^3)},
\end{align*}
and thus, by a Gronwall argument and our choice of the final datum, 
\[
\|\fn\|_{L^{\infty}((0,T);L^q(\R^3))} \le \|\chi\|_{L^1((0,T);L^q(\R^3))}.
\] 
Hence, there exists a subsequence $\{\nu_k\}_{k\in\N}$ that can be chosen as a subsequence of the one  found in Theorem \ref{thm:compactness} and an $\tilde f\in L^{\infty}((0,T);L^q(\R^3))$ such that
\[
f_{\nu_k}\to \tilde f\quad \mbox{weakly-$\star$ in }L^{\infty}((0,T);L^q(\R^3)).
\]
Since at the same time
\[
u_{\nu_k}\to u\quad \mbox{strongly in }L^2((0,T);L^2_{\rm{loc}}(\R^3)),
\]
and $q\ge2$, we find in the limit that $\tilde f$ solves the backward advection equation \eqref{tvp}, and thus, $\tilde f=f$ by uniqueness. In particular, the convergence result holds true for the subsequence from Theorem \ref{thm:compactness}. 
\end{proof}

We finally turn to the proof of the renormalization property.

\begin{proof}[Proof of Theorem \ref{thm:renorm}]
Let $\chi \in C_c^{\infty}((0,T)\times \H)$ be given and $\fn$ a solution to the backwards advection-diffusion equation considered in Lemma \ref{lem:backwards}.  From the statement of the lemma, it follows that $\{f_{\nu_k}\}_{k\in\N}$  converges to $f$ weakly-$\star$ in  $L^{\infty}((0,T);L^q(\R^3))$ for any $q\in (2,\infty)$. By using the advection-diffusion equation, this convergence can be upgraded to hold in $C([0,T],L^q_{\weak}(\R^3))$ for any $q\in (2,\infty)$, that is,
\begin{equation}\label{improved_conv}
\sup_{[0,T]} \int_{\R^3}(f_{\nu_k}(t)-f(t))\zeta\, dx\to 0\quad \forall \zeta\in L^{\tilde q}(\R^3),
\end{equation}
where $1/q+1/\tilde q=1$. Indeed, it is not difficult to obtain this result for smooth test functions and the full statement is obtained by standard approximation procedures.

Upon a standard approximation argument, $\fn$ can be considered as a test function in the distributional formulation of the vorticity formulation \eqref{xinu} of the Navier--Stokes equations. Thus
\begin{align*}
 \int_{\R^3} \fn(0)\xi_0\, dx  &= \int_0^T \int_{\R^3} \xin\left(\partial_t\fn +\un\cdot \grad\fn +\nu\left(\laplace \fn -\frac1r\partial_r \fn\right)\right)\, dxdt\\
 & = -\int_0^T \int_{\R^3} \xin\chi\, dxdt.
 \end{align*}
 As a consequence of Theorem \ref{thm:compactness}, Lemma \ref{lem:backwards} and \eqref{improved_conv}, we can pass to the limit in this identity and find  
 \[
 \int_{\R^3} f (0)\xi_0\, dx  + \int_0^T \int_{\R^3} \xi\chi\, dxdt =0.
 \]
 On the other hand, because $u$ satisfies the assumptions of Theorem \ref{thm:DP}, see also Remark \ref{rem}, there exists a unique distributional solution $\tilde \xi\in L^{\infty}((0,T);L^p(\R^3))$ to the transport equation \eqref{xi} with $u$ being the given solution to the Euler equations and with initial datum $\xi_0$. By Lemma \ref{s1-L1}, we then find that 
 \[
 \int_{\R^3} f (0)  \xi_0\, dx  + \int_0^T \int_{\R^3}\tilde  \xi\chi\, dxdt =0,
 \]
and thus,
\[
\int_0^T \int_{\R^3} (\xi-\tilde \xi)\chi\, dxdt=0.
\]
Because $\chi$ was arbitrarily fixed, we infer that $\tilde \xi = \xi$ almost everywhere, and thus, $\xi$ coincide almost everywhere with the renormalized solution $\tilde \xi$.
\end{proof}

\section{Energy conservation. Proof of Theorem \ref{thm:energy}}\label{sec:energy}

We now prove Theorem \ref{thm:energy}. Throughout this section, we thus suppose that $\on$ is nonnegative and has finite impulse. Moreover, we assume that $p>\frac32$ as in the assumption of the theorem. Notice that by interpolation between Lebesgue spaces, we may always suppose that $p\in \left(\frac32,2\right)$, which we will do from here on.

One of the main ingredients of the proof is the convergence of the kinetic energy that is established in the following lemma.

\begin{lemma}\label{L4}
Let $\{\nu_k\}_{k\in\N}$ be the subsequence found in Theorem \ref{thm:compactness}. Then it holds that
\[
\lim_{k\to \infty}\|u_{\nu_k}(t)\|_{L^2(\R^3)} = \|u(t)\|_{L^2(\R^3)}
\]
for any $t\in [0,T]$.
\end{lemma}
\begin{proof}

We have already seen in Theorem \ref{thm:compactness} that   $u_{\nu_k}$ converges to $u$ strongly in $C(0,T;L^2_{\loc}(\R^3))$. We have to turn this result into a global convergence result. In fact,  it is enough to show that
\begin{equation}\label{Riesz}
\sup_{k} \|\unuk(t)\|_{L^2(\R^3\setminus B_R(0))} \to 0 \quad\mbox{as }R\to \infty.
\end{equation}
Indeed, if \eqref{Riesz} holds true, given $\eps>0$, we can find a radius $R\ge 1 $ such that
\[
\sup_{k}\|\unuk(t,\cdot +h)\|_{L^2(\R^3\setminus B_{2R}(0))} \le \eps \quad \mbox{for any } |h|\le 1.
\]
Moreover, thanks to the strong convergence in $B_{2R}(0)$, we have that
\[
\sup_{k} \|\unuk(t) - \unuk(t,\cdot+h)\|_{L^2(B_R(0))} \le\eps \quad\mbox{for $|h|$ sufficiently small.}
\]
Combining both estimates, we find that
\[
\sup_{k} \|\unuk(t) - \unuk(t,\cdot+h)\|_{L^2(\R^3)} \le 3\eps \quad\mbox{for $|h|$ sufficiently small.}
\]
By Riesz' compactness criterion, the latter result together with \eqref{Riesz} and the standard energy estimate \eqref{NS3D-EI} imply strong convergence in $L^2(\R^3)$ for ant $t\in [0,T]$. 

We now give the argument for \eqref{Riesz}. For notational convenience, we write $u$ and $\nu$ instead of $\unuk$ and $\nu_k$. We consider a smooth cut-off function $\eta_R$ that is $1$ in $B_R = B_R(0)$ and $0$ outside $B_{2R} = B_{2R}(0)$. Testing the Navier--Stokes equations with $(1-\eta_R)^2 u$ and integrating by parts yields
\begin{align}
 \MoveEqLeft \frac{d}{dt}\frac12 \int (1-\eta_R)^2 |u|^2 \, dx + \nu \int (1-\eta_R)^2|\grad u|^2\, dx\label{Riesz0}\\
 & = \int (\eta_R-1)\grad \eta_R\cdot u|u|^2\, dx + 2\int (\eta_R-1)\grad\eta_R \cdot up\, dx\label{Riesz1}\\
 &\quad  +2 \nu \int(1-\eta_R)(\grad \eta_R\cdot \grad) u\cdot u\, dx.\label{Riesz2}
\end{align}

The error term in \eqref{Riesz2} is quite easily estimated. Indeed, using the Cauchy--Schwarz inequality together with the   elementary inequality $2ab\le \eps a^2 +\frac1\eps b^2$, we can absorb the  gradient term in \eqref{Riesz2} in the second term in \eqref{Riesz0} and we are left with an error term of the form $\frac{\nu}{R^2} \|u\|_{L^2(\R^3)}^2$. In view of the energy inequality for the Navier--Stokes equations, this term is obviously vanishing as $R\to \infty$ uniformly in $t$.

As a next step, we address the first error term in \eqref{Riesz1}. Using the properties of the cut-off function, this term  is  bounded as follows:
\begin{equation}\label{Riesz5}
\int (\eta_R-1)\grad \eta_R\cdot u|u|^2\, dx \lesssim \frac1R \int_{B_{2R}\setminus B_R} |u|^3\, d x\lesssim \int_{B_{2R}\setminus B_R} |u|^3 \, d(r,z).
\end{equation}
Here, we have used the same notation for both the ball in $\R^3$ and the one in $\H$. It should be clear from the situation, which one is considered.
We now use the Sobolev embedding in two dimensions and find
\[
\int_{B_{2R}\setminus B_R} |u|^3 \, d(r,z) \lesssim \left( R^{-\frac65} \int_{B_{2R}\setminus B_R} |u|^{\frac65}  \, d(r,z)+\int_{B_{2R}\setminus B_R} |\grad u|^{\frac65}  \, d(r,z) \right)^{\frac52}.
\]
Moreover, thanks to local Calderon--Zygmund estimates (which we can perform on the level of the standard Biot--Savart kernel in $\R^3$), we find that
\begin{equation}\label{Riesz3}
\int_{B_{2R}\setminus B_R} |\grad u|^{\frac65}  \, d(r,z) \lesssim  R^{-\frac65} \int_{B_{3R}\setminus B_{\frac R2}} |u|^{\frac65}  \, d(r,z)+\int_{B_{3R}\setminus B_{\frac{R}2}} |\omega|^{\frac65}  \, d(r,z)  .
\end{equation}
With regard to the first term in \eqref{Riesz3}, we notice that by Jensen's inequality and the energy inequality, it holds that
\begin{align*}
R^{-\frac65} \int_{B_{3R}\setminus B_\frac R2} |u|^{\frac65}  \, d(r,z) &\lesssim R^{-\frac25} \left(\int_{B_{3R}\setminus B_\frac R2} |u|^2 \, d(r,z)\right)^{\frac35}\\
& \lesssim R^{-1} \|u\|_{L^2(\R^3)}^{\frac65} \le R^{-1} \|u_0\|_{L^2(\R^3)}^{\frac65} ,
\end{align*}
and, thus, the first term in \eqref{Riesz3} vanishes as $R\to \infty$, uniformly in $t$. For the second term in \eqref{Riesz3}, we appeal to H\"older's inequality,
\[
\int_{B_{3R}\setminus B_{\frac{R}2}} |\omega|^{\frac65}  \, d(r,z) \le \left(\int_{B_{3R}\setminus B_{\frac{R}2}} |\omega|^{p}\, d(r,z)\right)^{\frac1{5(p-1)}} \left(\int_{B_{3R}\setminus B_{\frac{R}2}} |\omega|\, d(r,z)\right)^{\frac{5p-6}{5(p-1)}} .
\]
We can easily smuggle in some weights,
\[
\int_{B_{3R}\setminus B_{\frac{R}2}} |\omega|^{\frac65}  \, d(r,z) \lesssim R^{-\frac{9p-12}{5(p-1)}}  \|\xi\|_{L^p(\R^3)}^{\frac{p}{5(p-1)}}   \|\omega r^{2}\|_{L^1(\H)}^{\frac{5p-6}{5(p-1)}} .
\]
It remains to observe that the exponent on $R$ is negative because $p>\frac32>\frac{4}3$. Combining the above estimates, we conclude that
\begin{equation}\label{Riesz4}
 \lim_{R\to \infty}
\frac1R\int_{B_{2R}\setminus B_R} |u|^3 \, dx \sim   \lim_{R\to \infty}\int_{B_{2R}\setminus B_R} |u|^3 \, d(r,z)=0
\end{equation}
uniformly in $t$, and thus, in view of \eqref{Riesz5}, the first term in \eqref{Riesz1} vanishes.

We finally turn to the  term that involves the pressure, that is, the second term in \eqref{Riesz1}. Using the properties of the cut-off function and H\"older's inequality, we observe that
\[
\int (\eta_R-1)\grad\eta_R \cdot up\, dx \lesssim  \left(\frac1R\int_{B_{2R}\setminus B_R} |u|^3\, dx \right)^{\frac13} \left(\frac1R\int_{B_{2R}\setminus B_R} |p|^{\frac32}\, dx\right)^{\frac23}.
\]
In view of \eqref{Riesz4}, it is enough to show that the pressure term is bounded, in the sense that
\begin{equation}\label{Riesz6}
 \frac1R\int_{B_{2R}\setminus B_R}|p|^{\frac32}\, dx \lesssim 1
 \end{equation}
 uniformly in $R$, $\nu$ and $t$. 
To establish this estimate, we recall that $p$ solves the Poisson equation $-\laplace p = \grad^2 :u\otimes u$, and thus, we have that $p = \sum_{ij}\partial_i\partial_j G \ast (u_iu_j)$, where $G$ is the Newtonian potential in $\R^3$, $G(z)=\frac1{4\pi}\frac1{|z|}$. Let us write $f = G \ast (u_iu_j)$. The localized Calder\'on--Zygmund estimates yield
\[
R^{-\frac23}\|\grad^2 f\|_{L^{\frac32}(B_{2R}\setminus B_R)} \lesssim R^{-\frac23} \|u_iu_j\|_{L^{\frac32}(B_{3R}\setminus B_{\frac{R}2})} + R^{-\frac53} \|\grad f\|_{L^{\frac32}(B_{3R}\setminus B_{\frac{R}2})}.
\]
The first term is controlled thanks to \eqref{Riesz4}. (Notice that the exact value of the radii in \eqref{Riesz4} is not of importance but the scaling in $R$.) Now, since for any $m\in \left(\frac32,\infty\right)$, it holds that
\[
R^{-\frac53}\|\grad f\|_{L^{\frac32}(B_{3R})} \lesssim R^{\frac13-\frac3m}\|\grad f\|_{L^m(\R^3)}
\]
by H\"older's inequality, and $R^{\frac13-\frac3m}\to 0$ as $R\to\infty$ for $m<9$, it suffices to show that $\|\grad f\|_{L^m(\R^3)}$ is bounded for some $m\in \left(\frac32,9\right)$. Notice first that the standard Sobolev inequality in $\R^3$ yields 
$$\|\grad f\|_{L^m(\R^3)} \lesssim \|\grad^2 f\|_{L^s(\R^3)}$$
as long as $s = \frac{3m}{3+m} \in [1,3)$. By our choice of $m$, we have to restrict the range of admissible $s$ to the interval  $\left(1,\frac94\right)$.
Now we use the maximal regularity properties of the Laplacian, in the sense that 
$$\|\grad^2 f\|_{L^s(\R^3)} \lesssim \|\laplace f\|_{L^s(\R^3)} \lesssim \| u\|_{L^{2s}(\R^3)}^2.$$
In order to estimate the velocity field in $L^{2s}$, we use the Sobolev inequality in three dimensions and Calder\'on--Zygmund estimates for the gradient of the Biot--Savart kernel, 
$$\| u\|_{L^{2s}(\R^3)}\lesssim \|\grad u\|_{L^{\frac{6s}{2s+3}}(\R^3)}\lesssim \|\omega \|_{L^{\frac{6s}{2s+3}}(\R^3)}.$$
Let us now write  $q=\frac{6s}{2s+3}$ and notice that $q\in \left(\frac65,\frac{27}{15}\right)$ by our choice of $s$. Recall from our line of proof that we have to show the boundedness of $\|\omega\|_{L^q(\R^3)}$ for \emph{some} value of $q$ in the above interval. This is achieved via interpolation of the estimates in  Lemmas \ref{lem1} and \ref{L2}. Indeed, setting $\theta = \frac{q-1}{p-1}$ for some $q\in \left(\frac65,\frac{3}{2}\right)$ whose explicit value we will specify in a moment, we have that
\begin{eqnarray*}
\|\omega\|_{L^{q}(\R^3)}^q &=& 
\int |r\xi|^{(1-\theta)+\theta p}\,r\,d(r,z)\\
&=&\int r^{1+\theta(p-1)}|\xi|^{1-\theta}|\xi|^{\theta p}\, r\,d(r,z)\\
&\leq&\left(\int |\xi|r^{\frac{1+\theta(p-1)}{1-\theta}} r\,d(r,z)\right)^{1-\theta}\left(\int|\xi|^{p} r\,d(r,z)\right)^{\theta}\\
&=&\|\xi r^{\frac{q(p-1)}{p-q}}\|_{L^1(\R^3)}^{1-\theta}\|\xi\|_{L^p(\R^3)}^{\theta p}\,.
\end{eqnarray*}
where for the second to last inequality we used H\"older estimate with exponents $1/\theta$ and $\frac1{1-\theta}$ and in the last identity we used the definition of $\theta$.
We may now determine $q$ by requiring that $\frac{q(p-1)}{p-q}=2$, which yields $q= \frac{2p}{p+1}$. It remains to notice that this choice of $q$ is admissible because $\frac{2p}{p+1}\in\left(\frac65,\frac43\right)$ for any $p\in \left(\frac32,2\right)$ which we may assume here as explained in the introduction to this section. By using the a priori bounds in Lemmas \ref{lem1} and \ref{L2} we conclude that \eqref{Riesz6} holds uniformly in $\nu$ and $t$. We have thus established Lemma \ref{L4}.
%
%
%
%
%
%
%
%
%
%
%
\end{proof}

With these preparations, we are now in the position to prove Theorem \ref{thm:energy}. Our short proof is strongly inspired by \cite{CheskidovFilhoLopesShvydkoy16}.

\begin{proof}[Proof of Theorem \ref{thm:energy}]
In order to prove conservation of energy, we choose a subsequence as in Theorem \ref{thm:compactness}, which we will not relabel for notational convenience, and recall the energy identity in Lemma \ref{lem:balance}, which we rewrite as
\[
0\ge \|\unu(t)\|_{L^2(\R^3)}^2 - \|u_0\|_{L^2(\R^3)}^2   = -2\nu \int_0^t \|\grad_x \unu(s)\|^2_{L^2(\R^3)}\, ds.
\]
Thanks to Lemmas \ref{L3}, \ref{L2} and \ref{xi_improve}, we observe that
\[
\|\grad_x\un (s)\|_{L^2(\R^3)}^2  = \|r\xin(s)\|_{L^2(\R^3)}^2 \le \|r^2 \xin(s)\|_{L^1(\R^3)}\|\xin(s)\|_{L^{\infty}(\R^3)}\lesssim \left(\frac1{\nu s}\right)^{\frac{3}{2p}},
\]
and thus, the energy identity implies that
\[
0\ge \|\unu(t)\|_{L^2(\R^3)}^2 - \|u_0\|_{L^2(\R^3)}^2 \ge - C(\nu t)^{1 - \frac3{2p}},
\]
because $p>\frac32$. Sending $\nu$ to zero, we conclude that
\[
\lim_{\nu\to \infty} \|\unu(t)\|_{L^2(\R^3)} = \|u_0\|_{L^2(\R^3)},
\]
and the statement of the theorem follows upon applying Lemma \ref{L4}, in which the convergence of the kinetic energy is established.
\end{proof}


%

\section*{Appendix: Two  auxiliary inequalities}

We conclude this paper with two auxiliary inequalities, that are weighted versions of  standard Sobolev and interpolation inequalities. 

\begin{lemma}\label{weigthed_sobolev}
Let $1\le s\le t<\infty$ and $\alpha, \beta\in \R$ be such that
\[
\frac{2+\alpha}t = \frac{2-s+\beta}s\quad\mbox{and}\quad \alpha+t>0.
\]
Then
\[
\left(\int_{\H} |f|^t r^{\alpha}\, d(r,z)\right)^{\frac1t}\lesssim \left(\int_{\H} |\grad f|^s r^{\beta}\, d(r,z)\right)^{\frac1s},
\]
for any $f\in C_c^{\infty}(\H)$.
\end{lemma}

This estimate is proved, for instance, in \cite{Koch99}. We recall the argument for completeness.

\begin{proof}
\emph{Step 1.} We first treat the special case  $s=1$, and thus
\[
\frac{2+\alpha}t = 1+\beta.
\]
We set $\gamma = \frac{\alpha}t$ and let $g\in C_c^{\infty}(\H)$ be defined by $g(r,z) = f(2r,z)$ and $A = [R,2R]\times \R$ and $B = [R,4R]\times \R$ be two subsets of $\H$ for some $R>0$ fixed. By H\"older's inequality, we then have that
\[
\int_A (f-g)^2 r^{2\gamma}\, d(r,z) \le \int_R^{2R} \|r^{\gamma }(f-g)\|_{L^1(dz)} \|r^{\gamma} (f-g)\|_{L^{\infty}(dz)}\, dr.
\]
We now use the embedding $\dot W^{1,1}\subset L^{\infty}$, that holds true in one space dimension, in each variable. On the one hand, using the embedding in $r$ (in form of the fundamental theorem of calculus), we have 
\begin{align*}
\sup_{r\in(R,2R)} \|r^{\gamma }(f-g)\|_{L^1(dz)} &\le \sup_{r\in(R,2R)}r^{\gamma}\int_r^{2r} \|\partial_r f(\rho)\|_{L^1(dz)}d\rho \lesssim \int_B |\grad f|r^{\gamma }\, d(r,z).
\end{align*}
On the other hand, it holds that
\[
\int_0^{2R} \|r^{\gamma} (f-g)\|_{L^{\infty}(dz)} \lesssim \int_R^{2R} r^{\gamma} \|\partial_z(f-g)\|_{L^1(dz)}\, dr \lesssim \int_B |\grad f|r^{\gamma }\, d(r,z),
\]
where we have used the triangle inequality and a rescaling argument in the last inequality. Combining the previous three estimates, we find that
\begin{equation}\label{L2L1}
\left(\int_A (f-g)^2 r^{2\gamma}\, d(r,z)\right)^{\frac12} \lesssim \int_B |\grad f| r^{\gamma}\, d(r,z).
\end{equation}

Our next goal is the Hardy-type inequality
\begin{equation}\label{Hardy}
\int_A |f-g| r^{\gamma}\, d(r,z) \lesssim \int_B |\grad f| r^{\gamma+1}\, d(r,z),
\end{equation}
which holds true provided that $\gamma>-1$, and thus $\alpha+t>0$. It can be established as follows. Using the fundamental theorem again, we observe that
\[
\int_A |f-g|r^{\gamma} \, d(r,z) \le \int_R^{2R} r^{\gamma} \int_r^{2r} \|\partial_r f(\rho)\|_{L^1(dz)}\, d\rho dr \lesssim R^{\gamma+1} \int_B |\partial_r f|\, d(r,z),
\]
which implies \eqref{Hardy} because the prefactor $R^{\gamma+1}$ can be smuggled into the integrand.

Towards the weighted Sobolev inequality with $s=1$, we set $A_k = [2^k,2^{k+1}]\times \R$ and $B_k = [2^k,2^{k+2}]\times \R$ and estimate with the help of the triangle inequality
\begin{align*}
\left(\int_\H |f-g|^t r^{\alpha}\, d(r,z)\right)^{\frac1t} \le \sum_{k\in \Z} \left(\int_{A_k} |f-g|^t r^{\alpha}\, d(r,z)\right)^{\frac1t}
\end{align*}
Interpolation between Lebesgue spaces and an application of \eqref{L2L1} and \eqref{Hardy} yields
\begin{align*}
\MoveEqLeft\left(\int_{A_k} |f-g|^t r^{\alpha}\, d(r,z)\right)^{\frac1t} \\
&\le \left(\int_{A_k} |f-g|r^{\gamma}\, d(r,z)\right)^{\beta-\gamma}\left(\int_{A_k} (f-g)^2 r^{2\gamma}\, d(r,z)\right)^{\frac{1-\beta+\gamma}2}\\
& \lesssim \left(\int_{B_k} |\grad f|r^{\gamma+1}\, d(r,z)\right)^{\beta-\gamma}\left(\int_{B_k} |\grad f|r^{\gamma}\, d(r,z)\right)^{1-\beta+\gamma}\\
&\sim \int_{B_k} |\grad f| r^{\beta}\, d(r,z),
\end{align*}
because $\beta-\gamma = \frac2t-1\in[0,1]$ for $t\in[1,2]$. Summation over $k$ yields
\[
\left(\int_{\H} |f-g|^t r^{\alpha}\, d(r,z)\right)^{\frac1t}\le C\int_{\H} |\grad f| r^{\beta}\, d(r,z)
\]
for some universal constant $C$.
It remains to apply the triangle inequality and a change of variables to the effect that
\begin{align*}
\left(\int_{\H} |f|^t r^{\alpha}\, d(r,z)\right)^{\frac1t} &\le \left(\int_{\H} |g|^t r^{\alpha}\, d(r,z)\right)^{\frac1t} +C\int_{\H} |\grad f| r^{\beta}\, d(r,z)\\
& =  \frac1{2^{\frac{\alpha+1}t}}\left(\int_{\H} |f|^t r^{\alpha}\, d(r,z)\right)^{\frac1t} +C\int_{\H} |\grad f| r^{\beta}\, d(r,z).
\end{align*}
We can absorb the first term on the right-hand side by the left-hand side because $\alpha+1>0$ and obtain
\begin{equation}\label{case1}
\left(\int_{\H} |f|^{\frac{2+\alpha}{1+\beta}} r^{\alpha}\, d(r,z)\right)^{\frac{1+\beta}{2+\alpha}} \lesssim \int_{\H} |\grad f| r^{\beta}\, d(r,z).
\end{equation}

\emph{Step 2.} The general case  $s>1$ follows from the special case $s=1$. Indeed, choosing   $f = |h|^{\frac{2+\alpha}{(1+{\beta})t}}$ in \eqref{case1}, we find
\[
\left(\int_{\H} |f|^t r^{\alpha}\, d(r,z)\right)^{\frac1t} \lesssim \left(\int_{\H} |f|^{\frac{(1+\beta)t}{2+\alpha}-1} |\grad f| r^{\beta}\, d(r,z)\right)^{\frac{2+\alpha}{t(1+\beta)}},
\]
and the statement follows with the help of H\"older's inequality.
\end{proof}

We finally provide an interpolation inequality.

\begin{lemma}\label{weighted-interpolation}
Let $p\in (1,2]$ and $\lambda = \frac{3p-3}{7p-6}$. Then
\begin{align}
\MoveEqLeft \left(\int_{\H} |f|^4r\, d(r,z)\right)^{\frac14} \nonumber\\
&\lesssim \left(\int_{\H} |f|^2r\, d(r,z)\right)^{\frac\lambda2} \left(\int_{\H}|\grad f|^2 r\, d(r,z)\right)^{\frac14}\left(\int_{\H} |\grad f|^p r^{1-p}\, d(r,z)\right)^{\frac{1/2-\lambda}p}\label{I1}
\end{align}
for any $f\in C_c^{\infty}(\H)$.
\end{lemma}

\begin{proof}
\emph{Step 1}: It is enough to prove that
\begin{equation}\label{I2}
\left(\int_{\H} |f|^4r\, d(r,z)\right)^{\frac14} \lesssim \left(\int_{\H} |f|^2r\, d(r,z)\right)^{\frac\lambda2} \left(\int_{\H} |\grad f|^q r^{\beta}\, d(r,z)\right)^{\frac{1-\lambda}q},
\end{equation}
where $\beta = \frac{5p-4}{7p-4}$ and $q=\frac{16p-12}{7p-4}$.

Indeed, the statement in \eqref{I1} immediately follows from \eqref{I2} and H\"older's inquality. Let $a$ and $b$ be H\"older dual exponents given by $a = \frac{4(1-\lambda)}q = \frac{7p-4}{7p-6}$ and $b = \frac{4(1-\lambda)}{4(1-\lambda)-q} = \frac{7p-4}2$. We write and estimate
\begin{align*}
\int_{\H} |\grad f|^q r^{\beta}\, d(r,z) &  = \int_{\H} \left(|\grad f|^2 r\right)^\frac{q}{4-4\lambda} |\grad f|^{\frac{1-2\lambda}{2-2\lambda}q} r^{\beta - \frac{q}{4-4\lambda}}\, d(r,z)\\
&\le \left(\int_{\H}  \left(|\grad f|^2 r \right)^\frac{qa}{4-4\lambda}\, d(r,z)\right)^{\frac1a}\left(\int_{\H}  |\grad f|^{\frac{1-2\lambda}{2-2\lambda}qb} r^{\beta b- \frac{qb}{4-4\lambda}}\, d(r,z)\right)^{\frac1b}\\
& = \left(\int_{\H}  |\grad f|^2 r \, d(r,z)\right)^{\frac1a}\left(\int_{\H}  |\grad f|^p r^{1-p}\, d(r,z)\right)^{\frac1b}.
\end{align*} 
Now, plugging the resulting estimate into \eqref{I2} yields \eqref{I1}.

\emph{Step 2.} The interpolation inequality \eqref{I2} follows from the weighted Sobolev inequality from Lemma \ref{weigthed_sobolev} in the formulation
\begin{equation}\label{I3}
\left(\int_{\H} |f|^t r\, d(r,z) \right)^{\frac1t} \lesssim \left(\int_{\H} |\grad f|^sr^{\alpha} \, d(r,z)\right)^{\frac1s},
\end{equation}
where $t = \frac{16p-12}{7p-6}$, $s = \frac{16p-12}{13p-10}$ and $\alpha = \frac{11p-10}{13p-10}$ via a Ladyshenskaya-type argument. Notice that $t$, $s$ and $\alpha$ satisfy the dimensional condition
\[
t = \frac{3s}{2-s+\alpha}.
\]

Indeed, substituting  $|f|^\frac{4}t$ for $f$ in \eqref{I3} implies that
\[
\left(\int_{\H} |f|^4 r\, d(r,z)\right)^{\frac1t} \lesssim \left(\int_{\H} |f|^{(\frac4t-1)s }|\grad f|^s r^{\alpha}\, d(r,z)\right)^{\frac1s}.
\]
We now use H\"older's inequality with dual exponents $a = \frac{13p-10}{6p-6}$ and $b = \frac{13p-10}{7p-4}$ and get, since $r^{\alpha} = r^{\frac1a}  r^{\alpha -1 +\frac1b}$, that
\begin{align*}
\int_{\H} |f|^{(\frac4t-1)s }|\grad f|^s r^{\alpha}\, d(r,z) & \le\left( \int_{\H} |f|^{(\frac4t-1)sa } r\, d(r,z)\right)^{\frac1a} \left(\int_{\H} |\grad f|^{sb} r^{(\alpha-1+\frac1b)b}\, d(r,z)\right)^{\frac1b}\\
& = \left( \int_{\H} |f|^{2 } r\, d(r,z)\right)^{\frac1a} \left(\int_{\H} |\grad f|^{q} r^{\beta}\, d(r,z)\right)^{\frac1b}.
\end{align*}
Combining the previous two estimates, it is straightforward to deduce \eqref{I2}.
This completes the proof.
\end{proof}

\section*{Acknowledgement} 
CN thanks Stefano Spirito for pointing out references. CS acknowledges discussions with Helena Nussenzveig Lopes. This work is partially funded by the Deutsche Forschungsgemeinschaft (DFG, German Research Foundation) under Germany's Excellence Strategy EXC 2044 --390685587, Mathematics M\"unster: Dynamics--Geometry--Structure.

\bibliography{euler}

\begin{thebibliography}{10}

\bibitem{Abe18}
{\sc Abe, K.}
\newblock Vanishing viscosity limits for axisymmetric flows with boundary.
\newblock Preprint arXiv:1806.04811, 2018.

\bibitem{AbidiHmidiKeraani10}
{\sc Abidi, H., Hmidi, T., and Keraani, S.}
\newblock On the global well-posedness for the axisymmetric {E}uler equations.
\newblock {\em Math. Ann. 347}, 1 (2010), 15--41.

\bibitem{Ambrosio04}
{\sc Ambrosio, L.}
\newblock Transport equation and {C}auchy problem for {$BV$} vector fields.
\newblock {\em Invent. Math. 158}, 2 (2004), 227--260.

\bibitem{AmbrosioCrippa14}
{\sc Ambrosio, L., and Crippa, G.}
\newblock Continuity equations and {ODE} flows with non-smooth velocity.
\newblock {\em Proc. Roy. Soc. Edinburgh Sect. A 144}, 6 (2014), 1191--1244.

\bibitem{BerselliChiodaroli18}
{\sc Berselli, L., and Chiodaroli, E.}
\newblock {On the energy equality for the 3D Navier-Stokes equations}.
\newblock Preprint arXiv:1807.02667, 2018.

\bibitem{BohunBouchutCrippa16}
{\sc Bohun, A., Bouchut, F., and Crippa, G.}
\newblock Lagrangian solutions to the 2{D} {E}uler system with {$L^1$}
  vorticity and infinite energy.
\newblock {\em Nonlinear Anal. 132\/} (2016), 160--172.

\bibitem{Buckmaster15}
{\sc Buckmaster, T.}
\newblock Onsager's conjecture almost everywhere in time.
\newblock {\em Comm. Math. Phys. 333}, 3 (2015), 1175--1198.

\bibitem{BuckmasterDeLellisIsettSzekelyhidi15}
{\sc Buckmaster, T., {De Lellis}, C., Isett, P., and Sz{\'e}kelyhidi, Jr., L.}
\newblock Anomalous dissipation for {$1/5$}-{H}{\"o}lder {E}uler flows.
\newblock {\em Ann. of Math. (2) 182}, 1 (2015), 127--172.

\bibitem{BuckmasterDeLellisSzekelyhidi16}
{\sc Buckmaster, T., {De Lellis}, C., and Sz{\'e}kelyhidi, Jr., L.}
\newblock Dissipative {E}uler flows with {O}nsager-critical spatial regularity.
\newblock {\em Comm. Pure Appl. Math. 69}, 9 (2016), 1613--1670.

\bibitem{BuckmasterVicol17}
{\sc Buckmaster, T., and Vicol, V.}
\newblock Nonuniqueness of weak solutions to the navier-stokes equation.
\newblock Preprint arXiv:1709.10033, 2017.

\bibitem{ChaeImanuvilov98}
{\sc Chae, D., and Imanuvilov, O.~Y.}
\newblock Existence of axisymmetric weak solutions of the {$3$}-{D} {E}uler
  equations for near-vortex-sheet initial data.
\newblock {\em Electron. J. Differential Equations\/} (1998), No. 26, 17.

\bibitem{ChaeKim97}
{\sc Chae, D., and Kim, N.}
\newblock Axisymmetric weak solutions of the {$3$}-{D} {E}uler equations for
  incompressible fluid flows.
\newblock {\em Nonlinear Anal. 29}, 12 (1997), 1393--1404.

\bibitem{CheskidovConstantinFriedlanderShvydkoy08}
{\sc Cheskidov, A., Constantin, P., Friedlander, S., and Shvydkoy, R.}
\newblock Energy conservation and {O}nsager's conjecture for the {E}uler
  equations.
\newblock {\em Nonlinearity 21}, 6 (2008), 1233--1252.

\bibitem{CheskidovFilhoLopesShvydkoy16}
{\sc Cheskidov, A., Filho, M. C.~L., Lopes, H. J.~N., and Shvydkoy, R.}
\newblock Energy conservation in two-dimensional incompressible ideal fluids.
\newblock {\em Comm. Math. Phys. 348}, 1 (2016), 129--143.

\bibitem{CheskidovLuo18}
{\sc Cheskidov, A., and Luo, X.}
\newblock Energy equality for the {Navier-Stokes} equations in weak-in-time
  {O}nsager spaces.
\newblock Preprint arXiv:1802.05785, 2018.

\bibitem{ConstantinETiti94}
{\sc Constantin, P., E, W., and Titi, E.~S.}
\newblock Onsager's conjecture on the energy conservation for solutions of
  {E}uler's equation.
\newblock {\em Comm. Math. Phys. 165}, 1 (1994), 207--209.

\bibitem{CrippaNobiliSeisSpirito17}
{\sc Crippa, G., Nobili, C., Seis, C., and Spirito, S.}
\newblock Eulerian and {L}agrangian solutions to the continuity and {E}uler
  equations with {$L^1$} vorticity.
\newblock {\em SIAM J. Math. Anal. 49}, 5 (2017), 3973--3998.

\bibitem{CrippaSpirito15}
{\sc Crippa, G., and Spirito, S.}
\newblock Renormalized solutions of the 2{D} {E}uler equations.
\newblock {\em Comm. Math. Phys. 339}, 1 (2015), 191--198.

\bibitem{Danchin07}
{\sc Danchin, R.}
\newblock Axisymmetric incompressible flows with bounded vorticity.
\newblock {\em Uspekhi Mat. Nauk 62}, 3(375) (2007), 73--94.

\bibitem{DeLellisSzekelyhidi09}
{\sc {De Lellis}, C., and Sz{\'e}kelyhidi, J.~L.}
\newblock {The {E}uler equations as a differential inclusion}.
\newblock {\em Ann. of Math. (2) 170}, 3 (2009), 1417--1436.

\bibitem{DeLellisSzekelyhidi13}
{\sc {De Lellis}, C., and Sz{\'e}kelyhidi, Jr., L.}
\newblock Dissipative continuous {E}uler flows.
\newblock {\em Invent. Math. 193}, 2 (2013), 377--407.

\bibitem{Delort91}
{\sc Delort, J.-M.}
\newblock Existence de nappes de tourbillon en dimension deux.
\newblock {\em J. Amer. Math. Soc. 4}, 3 (1991), 553--586.

\bibitem{Delort92}
{\sc Delort, J.-M.}
\newblock Une remarque sur le probl{\`e}me des nappes de tourbillon
  axisym{\'e}triques sur {${\bf R}^3$}.
\newblock {\em J. Funct. Anal. 108}, 2 (1992), 274--295.

\bibitem{DiPernaLions89}
{\sc DiPerna, R.~J., and Lions, P.-L.}
\newblock Ordinary differential equations, transport theory and {S}obolev
  spaces.
\newblock {\em Invent. Math. 98}, 3 (1989), 511--547.

\bibitem{Eyink94}
{\sc Eyink, G.~L.}
\newblock Energy dissipation without viscosity in ideal hydrodynamics. {I}.
  {F}ourier analysis and local energy transfer.
\newblock {\em Phys. D 78}, 3-4 (1994), 222--240.

\bibitem{FengSverak15}
{\sc Feng, H., and {\v Sver{\'a}k}, V.~r.}
\newblock On the {C}auchy problem for axi-symmetric vortex rings.
\newblock {\em Arch. Ration. Mech. Anal. 215}, 1 (2015), 89--123.

\bibitem{Figalli08}
{\sc Figalli, A.}
\newblock Existence and uniqueness of martingale solutions for {SDE}s with
  rough or degenerate coefficients.
\newblock {\em J. Funct. Anal. 254}, 1 (2008), 109--153.

\bibitem{GallaySverak15}
{\sc Gallay, T., and {\v S}ver{\'a}k, V.}
\newblock Remarks on the cauchy problem for the axisymmetric navier-stokes
  equations.
\newblock {\em Confluentes Mathematici 7}, 2 (2015), 67--92.

\bibitem{HmidiZerguine09}
{\sc Hmidi, T., and Zerguine, M.}
\newblock Inviscid limit for axisymmetric {N}avier-{S}tokes system.
\newblock {\em Differential Integral Equations 22}, 11-12 (2009), 1223--1246.

\bibitem{Isett18}
{\sc Isett, P.}
\newblock A proof of {O}nsager's conjecture.
\newblock {\em Ann. of Math. (2) 188}, 3 (2018), 871--963.

\bibitem{JiuWuYang15}
{\sc Jiu, Q., Wu, J., and Yang, W.}
\newblock Viscous approximation and weak solutions of the 3{D} axisymmetric
  {E}uler equations.
\newblock {\em Math. Methods Appl. Sci. 38}, 3 (2015), 548--558.

\bibitem{JiuXin06}
{\sc Jiu, Q., and Xin, Z.}
\newblock On strong convergence to 3-{D} axisymmetric vortex sheets.
\newblock {\em J. Differential Equations 223}, 1 (2006), 33--50.

\bibitem{JiuXin04}
{\sc Jiu, Q.~S., and Xin, Z.~P.}
\newblock Viscous approximations and decay rate of maximal vorticity function
  for 3-{D} axisymmetric {E}uler equations.
\newblock {\em Acta Math. Sin. (Engl. Ser.) 20}, 3 (2004), 385--404.

\bibitem{Koch99}
{\sc Koch, H.}
\newblock {\em {Non-{E}uclidean singular integrals and the porous medium
  equation}}.
\newblock Habilitation thesis, Universit{\"a}t Heidelberg, Germany, 1999.

\bibitem{LeBrisLions04}
{\sc {Le Bris}, C., and Lions, P.-L.}
\newblock Renormalized solutions of some transport equations with partially
  {$W^{1,1}$} velocities and applications.
\newblock {\em Ann. Mat. Pura Appl. (4) 183}, 1 (2004), 97--130.

\bibitem{Leray34}
{\sc Leray, J.}
\newblock Sur le mouvement d'un liquide visqueux emplissant l'espace.
\newblock {\em Acta Math. 63}, 1 (1934), 193--248.

\bibitem{FilhoMazzucatoNussenzveig06}
{\sc {Lopes Filho}, M.~C., Mazzucato, A.~L., and {Nussenzveig Lopes}, H.~J.}
\newblock {Weak solutions, renormalized solutions and enstrophy defects in 2{D}
  turbulence}.
\newblock {\em Arch. Ration. Mech. Anal. 179}, 3 (2006), 353--387.

\bibitem{Onsager49}
{\sc Onsager, L.}
\newblock Statistical hydrodynamics.
\newblock {\em Nuovo Cimento (9) 6}, Supplemento, 2 (Convegno Internazionale di
  Meccanica Statistica) (1949), 279--287.

\bibitem{SaintRaymond94}
{\sc {Saint Raymond}, X.}
\newblock Remarks on axisymmetric solutions of the incompressible {E}uler
  system.
\newblock {\em Comm. Partial Differential Equations 19}, 1-2 (1994), 321--334.

\bibitem{Seis17}
{\sc Seis, C.}
\newblock A quantitative theory for the continuity equation.
\newblock {\em Ann. Inst. H. Poincar{\'e} Anal. Non Lin{\'e}aire 34}, 7 (2017),
  1837--1850.

\bibitem{Seis18}
{\sc Seis, C.}
\newblock Optimal stability estimates for continuity equations.
\newblock {\em Proc. Roy. Soc. Edinburgh Sect. A 148}, 6 (2018), 1279--1296.

\bibitem{Serrin63}
{\sc Serrin, J.}
\newblock The initial value problem for the {N}avier-{S}tokes equations.
\newblock In {\em Nonlinear {P}roblems ({P}roc. {S}ympos., {M}adison, {W}is.,
  1962)}. Univ. of Wisconsin Press, Madison, Wis., 1963, pp.~69--98.

\bibitem{Shinbrot74}
{\sc Shinbrot, M.}
\newblock The energy equation for the {N}avier-{S}tokes system.
\newblock {\em SIAM J. Math. Anal. 5\/} (1974), 948--954.

\bibitem{Stein}
{\sc Stein, E.~M.}
\newblock {\em {Singular integrals and differentiability properties of
  functions}}.
\newblock {Princeton Mathematical Series, No. 30}. Princeton University Press,
  Princeton, N.J., 1970.

\bibitem{UkhovskiiIudovich68}
{\sc Ukhovskii, M.~R., and Iudovich, V.~I.}
\newblock Axially symmetric flows of ideal and viscous fluids filling the whole
  space.
\newblock {\em J. Appl. Math. Mech. 32\/} (1968), 52--61.

\bibitem{VecchiWu93}
{\sc Vecchi, I., and Wu, S.~J.}
\newblock On {$L^1$}-vorticity for {$2$}-{D} incompressible flow.
\newblock {\em Manuscripta Math. 78}, 4 (1993), 403--412.

\bibitem{Wu10}
{\sc Wu, G.}
\newblock Inviscid limit for axisymmetric flows without swirl in a critical
  {B}esov space.
\newblock {\em Z. Angew. Math. Phys. 61}, 1 (2010), 63--72.

\end{thebibliography}
\bibliographystyle{acm}

\end{document}